\documentclass{amsart}

\usepackage[paperheight=11in, 
    paperwidth=8.5in,
    outer=1.25in,
    inner=1.25in,
    bottom=1.4in,
    top=1.4in]{geometry} 

\usepackage{amsmath, amsxtra, amsthm, amsfonts, amssymb}
\usepackage{mathrsfs}
\usepackage{graphicx}
\usepackage{url}
\usepackage{color}
\usepackage{bbm}
\usepackage{blkarray}
\usepackage{mathdots}
\usepackage{tikz-cd}

\usepackage[T1]{fontenc}
\usepackage{enumitem}
\usepackage{mathdots}
\usepackage{mathtools}
 \usepackage[hidelinks, backref=page]{hyperref} 
 \hypersetup{
    colorlinks,
    citecolor=magenta,
    filecolor=magenta,
    linkcolor=blue,
    urlcolor=black
}

\usepackage{cleveref}

\theoremstyle{definition}
\newtheorem{thm}{Theorem}[section]
\newtheorem{cor}[thm]{Corollary}
\newtheorem{lem}[thm]{Lemma}
\newtheorem{prop}[thm]{Proposition}
\newtheorem{defn}[thm]{Definition}

\newtheorem{eg}[thm]{Example}
\newtheorem{rem}[thm]{Remark}

\newtheorem{setup}[thm]{Setup}
\newtheorem{maintheorem}{Theorem}

\newcommand{\RR}{\mathbb R}

\newcommand{\CC}{\mathbb C}
\newcommand{\PP}{\mathbb P}

\newcommand{\be}{\mathbf{e}}
\newcommand{\bw}{\mathbf{w}}
\newcommand{\bv}{\mathbf{v}}
\newcommand{\bu}{\mathbf{u}}
\newcommand{\bs}{\mathbf{s}}
\newcommand{\bi}{\mathbf{i}}
\newcommand{\tinv}{\widetilde{\textnormal{inv}}}

\title{On two notions of total positivity for generalized partial flag varieties of classical Lie types}
\author{Grant Barkley, Jonathan Boretsky, Christopher Eur, Jiyang Gao}

\begin{document}
\begin{abstract}
For Grassmannians, Lusztig's notion of total positivity coincides with positivity of the Pl\"ucker coordinates.  This coincidence underpins the rich interaction between matroid theory, tropical geometry, and the theory of total positivity.  Bloch and Karp furthermore characterized the (type A) partial flag varieties for which the two notions of positivity similarly coincide.  We characterize the symplectic (type C) and odd-orthogonal (type B) partial flag varieties for which Lusztig's total positivity coincides with Pl\"ucker positivity.
\end{abstract}

\maketitle

\vspace{-10pt}

\section{Introduction}
Let $n$ be a positive integer, and denote $[n] = \{1, \dotsc, n\}$.
The \emph{totally positive} (resp.\ \emph{nonnegative}) \emph{part} $GL_n^{>0}$ (resp.\ $GL_n^{\geq 0}$) of the general linear group $GL_n$ consists of the real invertible matrices whose minors are all positive (resp.\ nonnegative).
The study of these spaces traces back to \cite{ASW, Loewnerpositive, Whitneypositivity}.
Lusztig generalized this notion of total positivity to an arbitrary connected reductive ($\RR$-split) algebraic group $G$ and its partial flag varieties $G/P$ \cite{lusztig1994positivity,lusztig1998positivity}.
The study of total positivity has since been a nexus for fruitful interactions between algebraic geometry, representation theory, combinatorics, and physics \cite{physicsbook,FZsurvey,postnikov2006,Fomincluster,GWpos}.

\medskip
Underpinning such fruitful interactions is the interplay between ``parametric'' and ``implicit'' descriptions of total positivity.
The original definitions are ``parametric'' in nature: Lusztig defined the \emph{Lusztig positive} (resp.\ \emph{Lusztig nonnegative}) part $G^{>0}$ (resp.\ $G^{\geq 0}$) of $G$ as a semigroup in $G$ generated by certain elements (see Section~\ref{subsect:Luspositivity}).
For a parabolic subgroup $P\subset G$, the Lusztig positive (resp.\ Lusztig nonnegative) part $(G/P)^{>0}$ (resp.\ $(G/P)^{\geq 0}$) of the partial flag variety $G/P$ is then defined as the image of (resp.\ the closure of the image of) $G^{>0}$ under the projection map $G\to G/P$.  Marsh and Rietsch gave a combinatorial parametrization of $(G/P)^{\geq 0}$ in terms of its Deodhar cells \cite{marshrietsch2004}.

\medskip
On the other hand, one may seek an ``implicit'' description of Lusztig positivity for $G/P$ in terms of positivity of suitably natural coordinates on $G/P$.
As a motivating example, consider the \emph{Grassmannian} $\operatorname{Gr}_{k;n}$ of $k$-dimensional subspaces in $\RR^n$.
Its Pl\"ucker coordinates allow one to consider the \emph{Pl\"ucker positive} (resp.\ \emph{nonnegative}) \emph{part} of $\operatorname{Gr}_{k;n}$, defined as
\[
\text{$\operatorname{Gr}_{k;n}^{\Delta > 0}$ (resp.\ $\operatorname{Gr}_{k;n}^{\Delta \geq 0}$)} := \left\{L \subseteq \RR^n \ \middle| \ \begin{matrix}\text{$L$ is the row-span of a real $k\times n$ matrix all of whose}\\  \text{maximal minors are positive (resp.\ nonnegative)}\end{matrix}\right\}.
\]
Lam \cite{Lam2014} and, independently, Talaska and Williams \cite{talaska2013network} showed that $\operatorname{Gr}_{k;n}^{> 0} = \operatorname{Gr}_{k;n}^{\Delta > 0}$ and $\operatorname{Gr}_{k;n}^{\geq 0} = \operatorname{Gr}_{k;n}^{\Delta \geq 0}$.
More generally, for a subset $K = \{k_1 < \dotsb < k_j\} \subseteq [n-1]$, one may consider the (type A) partial flag variety
\[
\operatorname{Fl}_{K ; n} := \{\text{flags of subspaces $L_\bullet = (L_1 \subseteq \dotsb \subseteq L_j)$ with $\dim L_i = k_i$ for all $i = 1, \dots, j$}\}.
\]
Under the natural embedding $\operatorname{Fl}_{K ; n} \hookrightarrow \prod_{i = 1}^j \operatorname{Gr}_{k_i;n}$,
its \emph{Pl\"ucker positive} (resp.\ \emph{nonnegative}) part is defined as the intersection
\[
\text{$\operatorname{Fl}_{K ; n}^{\Delta>0}$ (resp.\ $\operatorname{Fl}_{K ; n}^{\Delta \geq 0}$)} := \operatorname{Fl}_{K ; n} \cap  \prod_{i = 1}^j \operatorname{Gr}_{k_i;n}^{\Delta >0} \text{ (resp.\ $\prod_{i = 1}^j \operatorname{Gr}_{k_i;n}^{\Delta \geq 0}$)}.
\]
Bloch and Karp \cite{blochkarp2023} showed the following.  The second author independently showed a similar result in the case of $K =[n-1]$ \cite{boretsky2023totally}.

\begin{thm}\cite[Theorem 1.1]{blochkarp2023}\label{thm:BK}
The following are equivalent for a subset $K\subseteq[n-1]$:
\begin{enumerate}[label = (\arabic*)]
\item $\operatorname{Fl}_{K ; n}^{>0} = \operatorname{Fl}_{K; n}^{\Delta>0}$,
\item $\operatorname{Fl}_{K; n}^{\geq 0} = \operatorname{Fl}_{K; n}^{\Delta\geq 0}$, and
\item $K$ consists of consecutive integers.
\end{enumerate}
\end{thm}

In summary, these results establish the coincidence of Lusztig's positivity and Pl\"ucker positivity for Grassmannians, and more generally for partial flag varieties $\operatorname{Fl}_{K;n}$ with consecutive $K$.
This coincidence supports the rich interaction between matroid theory, tropical geometry, and total positivity \cite{ardilarinconwilliams, postnikov2006, SpeyerWilliams2005, SpeyerWilliams2021, boretskyeurwilliams, JLLO, PSBW23}.
Here, with a view towards the theory of Coxeter matroids \cite{BGW2003} and their tropical geometry \cite{Rincon2011, BallaOlarte2023}, we characterize the partial flag varieties of the symplectic group $\operatorname{Sp}_{2n}$ (type C) and the odd-orthogonal group $\operatorname{SO}_{2n+1}$ (type B) for which Lusztig's positivity coincides with Pl\"ucker positivity.

\medskip
Let $\be_i$ denote the $i$-th standard basis vector in a coordinate space, and $\be_i^*$ its dual.
For type C, endow $\RR^{2n}$ with the symplectic bilinear form $\omega = \sum_{i = 1}^{n} (-1)^i \be_i^* \wedge \be_{2n+1 - i}^*$.
For type B, endow $\RR^{2n+1}$ with the symmetric bilinear form $Q = \sum_{i = 1}^{n+1} (-1)^i \be_i^* \cdot \be_{2n+2 - i}^*$.
Let $\operatorname{Sp}_{2n}$ and $\operatorname{SO}_{2n+1}$ be the linear groups preserving the bilinear forms $\omega$ and $Q$, respectively.
Explicitly, we have
\begin{align*}
\mathrm{Sp}_{2n}(\RR) & :=\{A\in \mathrm{SL}_{2n}(\RR)| A^tEA = E\},\quad\text{where\quad}E=\left[\begin{smallmatrix}
    &&&&1\\&&&-1&\\&&\iddots&&\\&1&&&\\-1&&&&
\end{smallmatrix}\right], \quad\text{and}\\
\operatorname{SO}_{2n+1}(\RR) &:= \{A\in \mathrm{SL}_{2n+1}(\RR) | A^tE'A = E'\},\quad\text{where}\quad E'= \begin{bmatrix} & -1 \\ E & \end{bmatrix} = \left[\begin{smallmatrix}
    &&&&-1\\&&&1&\\&&\iddots&&\\&1&&&\\-1&&&&
\end{smallmatrix}\right].
\end{align*}
The additional choice of pinnings required for defining their Lusztig positive parts $\operatorname{Sp}_{2n}^{>0}$ and $\operatorname{SO}_{2n+1}^{>0}$ is described in Section~\ref{sect:pinning}. Recall that a subspace of a vector space with a symmetric or alternating form is \emph{isotropic} if the restriction of the form to the subspace is trivial. The partial flag varieties of these groups have the following description in terms of isotropic subspaces (see Section~\ref{sect:partialflag} for details).

For $K\subseteq [n]$, let
\begin{align*}
\operatorname{SpFl}_{K;2n} &:= \{ L_\bullet \in \operatorname{Fl}_{K;2n} : \text{each subspace $L_i$ in the flag $L_\bullet$ is isotropic with respect to $\omega$}\}, \quad\text{and}\\
\operatorname{SOFl}_{K;2n+1} &:= \{ L_\bullet \in \operatorname{Fl}_{K;2n+1} : \text{each subspace $L_i$ in the flag $L_\bullet$ is isotropic with respect to $Q$}\}.
\end{align*}
We define their \emph{Pl\"ucker positive} (resp.\ \emph{nonnegative}) parts as the intersections
\begin{align*}
\operatorname{SpFl}_{K;2n}^{\Delta >0}\text{ (resp.\ $\operatorname{SpFl}_{K; 2n}^{\Delta \geq 0}$)} &:= \operatorname{SpFl}_{K;2n} \cap \operatorname{Fl}_{K; 2n}^{\Delta > 0} \text{ (resp.\ $\operatorname{Fl}_{K; 2n}^{\Delta \geq 0}$)}, \quad\text{and}\\
\operatorname{SOFl}_{K;2n+1}^{\Delta >0}\text{ (resp.\ $\operatorname{SOFl}_{K; 2n+1}^{\Delta \geq 0}$)} &:= \operatorname{SOFl}_{K;2n+1} \cap \operatorname{Fl}_{K; 2n+1}^{\Delta > 0} \text{ (resp.\ $\operatorname{Fl}_{K; 2n+1}^{\Delta \geq 0}$)}.
\end{align*}
Our main theorem is as follows.

\begin{maintheorem}\label{thm:main}
In type C, for $n \geq 2$ and a subset $K\subseteq [n]$, the following are equivalent:
\begin{enumerate}[label = (\arabic*)]
\item\label{mainthm:pos} $\operatorname{SpFl}_{K ; 2n}^{>0} = \operatorname{SpFl}_{K; 2n}^{\Delta>0}$,
\item\label{mainthm:nonneg} $\operatorname{SpFl}_{K; 2n}^{\geq 0} = \operatorname{SpFl}_{K; 2n}^{\Delta\geq 0}$, and
\item\label{mainthm:consec} $K = \{k, k +1, \dotsc, n\}$ for some $1\leq k\leq n$.
\end{enumerate}
In type B, for $n \geq 3$ and a subset $K\subseteq [n]$, the following are equivalent:
\begin{enumerate}[label = (\arabic*)]
\item $\operatorname{SOFl}_{K ; 2n+1}^{>0} = \operatorname{SOFl}_{K; 2n+1}^{\Delta>0}$,
\item $\operatorname{SOFl}_{K; 2n+1}^{\geq 0} = \operatorname{SOFl}_{K; 2n+1}^{\Delta\geq 0}$, and
\item $K = \{k, k +1, \dotsc, n\}$ for some $1\leq k\leq n$.
\end{enumerate}
In type B, when $n=2$, the statements (1) and (2) hold for all partial flag varieties $\operatorname{SOFl}_{K;5}$.
\end{maintheorem}

In our proof of \Cref{thm:main}, we identify a general condition (see  \hyperref[condition1]{($\dagger$)} in Definition~\ref{condition1}) that implies \ref{mainthm:pos}$\implies$\ref{mainthm:nonneg}, and show that a mild strengthening of the condition (see \hyperref[prop:strongdagger]{($\dagger1'$)}) implies \ref{mainthm:consec}$\implies$\ref{mainthm:pos} for the $K = [n]$ case.
In contrast, we show that for $\operatorname{SO}_{2n}$ (type D), there exists no pinning for which the general condition \hyperref[condition1]{($\dagger$)} is satisfied (Proposition~\ref{prop:typeDpinning}).

\subsection*{Previous works}
Karpman showed that the statements (1) and (2) hold for Lagrangian Grassmannians, i.e.\ $\operatorname{SpFl}_{n;2n}$ \cite{karpman2018}. Theorem A implies that the methods there cannot generalize to $\operatorname{SpFl}_{k;2n}$ for $k\neq n$. For an explanation of why, see Remark~\ref{rem:karpman}.

For a general reductive ($\RR$-split) algebraic group $G$ of simply-laced type, Lusztig showed that Lusztig positivity for a partial flag variety $G/P$ coincides with positivity of the coordinates from the canonical basis of a sufficiently large irreducible representation of $G$ \cite{lusztig1998positivity}.
However, due to the ``sufficiently large'' condition, this does not recover any of the aforementioned results of Lam, Talaska--Williams, Bloch--Karp, or Karpman.
Chevalier \cite[Example 10.2]{chevaliercounterexample} gave an example showing that the ``sufficiently large'' condition cannot be removed; we verify the example explicitly in Section~\ref{sect:typeD}.

\subsection*{Organization}
\Cref{sect:pinning} provides background on pinnings and establishes the conventions.
\Cref{sec:compatibility} describes how the pinnings for $\operatorname{Sp}_{2n}$ and $\operatorname{SO}_{2n+1}$ are compatible with the standard pinnings of $GL_{2n}$ and $GL_{2n+1}$.
\Cref{sect:Luspositivity} defines Lusztig positivity and nonnegativity, and proves the implications \ref{mainthm:consec}$\implies$\ref{mainthm:pos}$\implies$\ref{mainthm:nonneg} in \Cref{thm:main}. \Cref{sect:constructions} provides explicit examples that establish \ref{mainthm:nonneg}$\implies$\ref{mainthm:consec}, thereby completing the proof of \cref{thm:main}. \Cref{sect:typeD} discusses difficulties that arise for flag varieties of type D.
\Cref{sect:MRparameterization} presents two alternate proofs of the implication \ref{mainthm:consec}$\implies$\ref{mainthm:pos} in \Cref{thm:main} by establishing further properties of the embeddings $\operatorname{Sp}_{2n}\hookrightarrow GL_{2n}$ and $\operatorname{SO}_{2n+1} \hookrightarrow GL_{2n+1}$ that may be of independent interest.

\subsection*{Acknowledgements}
The authors would like to thank Pavel Galashin, Steven Karp, Konstanze Rietsch, and Lauren Williams for helpful communications regarding this work. The first, second, and fourth authors were supported in part by NSF grant DMS-2152991. The second author was also supported by the Natural Sciences and Engineering Research Council of Canada (NSERC). Le deuxième auteur a été aussi financé par le Conseil de recherches
en sciences naturelles et en génie du Canada (CRSNG) [Ref. no. 557353-2021]. The third author is supported by NSF grant DMS-2246518.

\section{Pinnings}\label{sect:pinning}

Let $G$ be a connected, reductive, $\RR$-split linear algebraic group.  We often identify $G$ with its $\RR$-valued points.
A \emph{pinning} of $G$ is an additional set of choices for $G$ that is part of the input data for the definition of Lusztig positivity for $G$.
We set up notations, and describe our choice of pinnings for $\operatorname{Sp}_{2n}$ and $\operatorname{SO}_{2n+1}$ in this section.

\subsection{Generalities and notations}
Fix a split maximal torus $T$ in $G$, and let $X$ be the character lattice of $T$.
Let $\Phi \subset X$ be the set of roots of the corresponding root system.
Fix a system of positive roots $\Phi^+$, and let $B_+$ be the corresponding Borel subgroup of $G$.
Let $B_-$ be the opposite Borel subgroup such that $B_+\cap B_{-} =T$.
Let $U_+$ and $U_-$ be the unipotent radicals of $B_+$ and of $B_-$, respectively.
Let $I$ be an indexing set for the set $\{\alpha_i : i\in I\}$ of simple roots in $\Phi^+$.
For every $i\in I$, fix a homomorphism $\phi_i:SL_2\rightarrow G$ such that in the induced map $\mathfrak{sl}_2 \to \mathfrak g$ of Lie algebras, the element $\left[\begin{smallmatrix} 0 & 1\\ 0 & 0\end{smallmatrix}\right]\in \mathfrak{sl}_2$ maps to a generator of the root space in $\mathfrak g$ of weight $\alpha_i$.
We then define homomorphisms $x_i: \RR \to U_+$, $y_i: \RR \to U_-$, and $\chi_i: \RR^* \to T$ by 
\[
x_i(m) :=\phi_i\left(\begin{bmatrix}
        1 & m\\
        0 & 1
    \end{bmatrix}\right),
\quad
y_i(m) :=\phi_i\left(\begin{bmatrix}
        1 & 0\\
        m & 1
    \end{bmatrix}\right),
\quad\text{and}\quad
\chi_i(t) := \phi_i\left(\begin{bmatrix}
        t & 0\\
        0 & t^{-1}
    \end{bmatrix}\right).
\]
One may observe that the choices made so far for the triple ($T$, $B_+$, $\{\phi_i\}_{i\in I}$) is equivalent to a choice of a set $\{(e_i,f_i)\}_{i\in I}$ of Chevalley generators of the Lie algebra $\mathfrak g$.

\begin{defn}
The data $(T, B_+, B_-, \{x_i\}_{i\in I}, \{y_i\}_{i\in I})$ is called a \emph{pinning} for $G$.
\end{defn}

When multiple groups are in play, we write superscripts of the root system name, for example $T^{\Phi}$, $s_i^\Phi$, and $y_i^\Phi$, to distinguish between the notations for pinnings of different groups.

\medskip
A pinning of $G$ identifies the reflection group $W$ of the root system $\Phi$ with the Weyl group $N_G(T)/T$, as follows.
For each $i\in I$, the simple reflection $s_i\in W$ is identified with $\dot{s}_iT$ where
\[
\dot{s}_i := \phi_i\left(\left[\begin{matrix}
    0 & -1 \\
    1 & 0
\end{matrix}\right]\right).
\]
Given an expression $\mathbf{w}=s_{i_1}s_{i_2}\cdots s_{i_l}$, we denote $\dot{\mathbf{w}}=\dot{s_{i_1}}\dot{s_{i_2}}\cdots \dot{s_{i_l}}$. 

\medskip
For a sequence $\mathbf i = (i_1, \dotsc, i_\ell)$ with entries in the indexing set $I$ of the simple roots, we denote by $\mathbf s_{\mathbf i}$ the element
\[
\mathbf s_{\mathbf i} := s_{i_1}\dotsm s_{i_\ell} \in W.
\]
When clear from context, we use $\mathbf s_{\mathbf i}$ to denote also the word $(s_{i_1}, \dotsc, s_{i_\ell})$.
Define the function $\mathbf y_{\mathbf i}: \RR^{\ell} \to G$ by
\[
\mathbf y_{\mathbf i}(a_1, \dotsc, a_\ell) := y_{i_1}(a_1) \dotsm y_{i_\ell}(a_\ell),
\]
and similarly define $\mathbf x_{\mathbf i}$, $\boldsymbol \chi_{\mathbf i}$, and $\dot{\mathbf s}_{\mathbf i}$.
The length $\ell$ of the sequence $\mathbf i$ is denoted $|\mathbf i|$.

\medskip
In type $A_{n-1}$, when $G=GL_n$, we use the \emph{standard pinning} $(T^{A}, B_+^{A}, B_-^{A}, \{x_i^{A}\}_{i\in [n-1]}, \{y_i^{A}\}_{i\in [n-1]})$, defined as follows.
The torus $T^{A}$ consists of diagonal matrices with non-zero entries on the diagonal. The Borels $B_+^{A}$ and $B_-^{A}$ consist of upper and lower triangular invertible matrices, respectively. The set of simple roots is $\{\be_1-\be_2, \dotsc, \be_{n-1} -\be_n\}$. Accordingly, for each $i\in [n-1]$, the maps $\phi_i^{A}$ are given by
\begin{equation*}
    \phi_i^{A}\left(\begin{pmatrix}
        a & b\\
        c&d
    \end{pmatrix}\right) :=\begin{blockarray}{ccccccc}
& &  & i & i+1 & & \\
\begin{block}{c(cccccc)}
  & 1 &  &  & & & \\
  & & \ddots &  & &  & \\
  i& &  & a & b &  &  \\
  i+1& &  &c  & d &  &  \\
  & &  &  &  & \ddots  & \\
  &&&&&&1\\ 
\end{block} 
\end{blockarray}\hspace{5pt}, 
\end{equation*}
where unmarked off-diagonal matrix entries are $0$.  The Weyl group is the permutation group $\mathfrak S_n$ on $[n]$ with $s_i^A$ the transposition $(i\ i+1)$.

\subsection{Pinnings of \texorpdfstring{$\operatorname{Sp}_{2n}$}{Sp2n} and \texorpdfstring{$\operatorname{SO}_{2n+1}$}{SO2n+1}}
We provide explicit descriptions of the pinnings of $\operatorname{Sp}_{2n}$ and $\operatorname{SO}_{2n+1}$ in \S\ref{sssec:Cpinning} and \S\ref{sssec:Bpinning}, respectively.
One may verify that they are indeed valid pinnings from \cite{BilleyLakshmibai}, which provides explicit descriptions of Chevalley generators of the Lie algebras $\mathfrak{sp}_{2n}$ and $\mathfrak{so}_{2n+1}$.
We record the key properties of these pinnings that we will use in Section~\ref{sec:compatibility}.

\subsubsection{Type C pinning}\label{sssec:Cpinning}
The pinning $(T^{C_{n}}, B_+^{C_{n}}, B_-^{C_{n}}, \{x_i^{C_{n}}\}_{i\in [n]}, \{y_i^{C_{n}}\}_{i\in [n]})$ of $\operatorname{Sp}_{2n}$ is defined as follows.
The torus $T^{C_{n}}$ consists of matrices 

\begin{equation*}
\begin{pmatrix}
    t_1&&&&&&&\\
    &t_2&&&&&&\\
    &&\ddots&&&&&\\
    &&&t_n&&&&\\
    &&&&t_n^{-1}&&&\\
    &&&&&\ddots&&\\
    &&&&&&t_2^{-1}&\\
    &&&&&&& t_1^{-1}
\end{pmatrix},
\end{equation*}
where $t_i\in \mathbb{R}^*$ for $i\in [n]$ and all off-diagonal entries are $0$.
The Borels $B_+^{C_{n}}$ and $B_-^{C_{n}}$ consist of upper and lower triangular matrices in $\operatorname{Sp}_{2n}$, respectively, with nonzero entries on the diagonal. The set of simple roots is $\{\be_1-\be_2, \dotsc, \be_{n-1}-\be_n, 2\be_n\}$. Accordingly, for $i\in [n-1]$, the map $\phi_i^{C_{n}}$ is given by
\begin{equation*}
    \phi_i^{C_{n}}\left(\begin{pmatrix}
        a&b\\
        c&d
    \end{pmatrix}\right)=\begin{blockarray}{cccccccccc}
& &  & i & i+1 & & 2n-i& 2n-i+1 & &\\
\begin{block}{c(ccccccccc)}
  & 1 &  &  & & & & & &\\
  & & \ddots &  & &  & & & & \\
  i& &  & a & b &  & & & & \\
  i+1& &  &c  & d &  &  & & & \\
  & &  &  &  & \ddots  & & & &  \\
  2n-i&&&&&&a&b&&\\
  2n-i+1&&&&&&c&d&&\\
 &&&&&&&&\ddots&\\
 &&&&&&&&&1\\
\end{block} 
\end{blockarray}\hspace{5pt}, 
\end{equation*}
where all unmarked off-diagonal entries are $0$.
For $i=n$, the map $\phi_n^{C_n}$ is given by
\begin{equation*}
    \phi_n^{C_{n}}\left(\begin{pmatrix}
        a&b\\
        c&d
    \end{pmatrix}\right)=\begin{blockarray}{ccccccc}
& &  & n & n+1 & & \\
\begin{block}{c(cccccc)}
  & 1 &  &  & & & \\
  & & \ddots &  & &  & \\
  n& &  & a & b &  &  \\
  n+1& &  &c  & d &  &  \\
  & &  &  &  & \ddots  & \\
  &&&&&&1\\ 
\end{block} 
\end{blockarray}\hspace{5pt}, 
\end{equation*}
where all unmarked off-diagonal entries are $0$.

\subsubsection{Type B pinning}\label{sssec:Bpinning}

The pinning $(T^{B_{n}}, B_+^{B_{n}}, B_-^{B_{n}}, \{x_i^{B_{n}}\}_{i\in [n]}, \{y_i^{B_{n}}\}_{i\in [n]})$ of $\operatorname{SO}_{2n+1}$ is defined as follows.
The torus $T^{B_{n}}$ consists of matrices 
\begin{equation*}
\begin{pmatrix}
    t_1&&&&&&&&\\
    &t_2&&&&&&&\\
    &&\ddots&&&&&&\\
    &&&t_n&&&&&\\
    &&&&1&&&&\\
    &&&&&t_n^{-1}&&&\\
    &&&&&&\ddots&&\\
    &&&&&&&t_2^{-1}&\\
    &&&&&&&& t_1^{-1}
\end{pmatrix},
\end{equation*}
\noindent where $t_i\in \mathbb{R}^*$ for $i\in [n]$ and all off-diagonal entries are $0$.
The Borels $B_+^{B_{n}}$ and $B_-^{B_{n}}$ consist of upper and lower triangular matrices in $\operatorname{SO}_{2n+1}$, respectively, with nonzero entries on the diagonal. The set of simple roots is $\{\be_1-\be_2, \dotsc, \be_{n-1}-\be_n, \be_n\}$.
Accordingly, for $i\in [n-1]$, the map $\phi_i^{B_{n}}$ is given by

\begin{equation*}
    \phi_i^{B_{n}}\left(\begin{pmatrix}
        a&b\\
        c&d
    \end{pmatrix}\right)=\begin{blockarray}{cccccccccc}
& &  & i & i+1 & & 2n-i+1& 2n-i+2 & &\\
\begin{block}{c(ccccccccc)}
  & 1 &  &  & & & & & &\\
  & & \ddots &  & &  & & & & \\
  i& &  & a & b &  & & & & \\
  i+1& &  &c  & d &  &  & & & \\
  & &  &  &  & \ddots  & & & &  \\
  2n-i+1&&&&&&a&b&&\\
  2n-i+2&&&&&&c&d&&\\
 &&&&&&&&\ddots&\\
 &&&&&&&&&1\\
\end{block} 
\end{blockarray}\hspace{5pt}, 
\end{equation*}
where all unmarked off-diagonal entries are $0$.
For $i=n$, the map $\phi_n^{B_n}$ is determined by 

\begin{equation*}
    \phi_n^{B_{n}}\left(\begin{pmatrix}
        t&0\\
        0&t^{-1}
    \end{pmatrix}\right)=\begin{blockarray}{cccccccc}
& &  & n & n+1 & n+2 & & \\
\begin{block}{c(ccccccc)}
  & 1 &  &  & & & & \\
  & & \ddots &  & & &  & \\
  n& &  & t^2 &  & &  &  \\
  n+1& &  &  & 1 &  & &  \\
  n+2 & & & & & t^{-2} & \\
  & &  &  & & & \ddots  & \\
  &&&&&&&1\\ 
\end{block} 
\end{blockarray}\hspace{5pt},
\end{equation*}

\begin{equation*}
    \phi_n^{B_{n}}\left(\begin{pmatrix}
        1&m\\
        0&1
    \end{pmatrix}\right)=\begin{blockarray}{cccccccc}
& &  & n & n+1 & n+2 & & \\
\begin{block}{c(ccccccc)}
  & 1 &  &  & & & & \\
  & & \ddots &  & & &  & \\
  n& &  & 1 & \sqrt{2}m & m^2&  &  \\
  n+1& &  &  & 1 & \sqrt{2}m & &  \\
  n+2 & & & & & 1 & \\
  & &  &  & & & \ddots  & \\
  &&&&&&&1\\ 
\end{block} 
\end{blockarray}\hspace{5pt},
\end{equation*}
and

\begin{equation*}
    \phi_n^{B_{n}}\left(\begin{pmatrix}
        1&0\\
        m&1
    \end{pmatrix}\right)=\begin{blockarray}{cccccccc}
& &  & n & n+1 & n+2 & & \\
\begin{block}{c(ccccccc)}
  & 1 &  &  & & & & \\
  & & \ddots &  & & &  & \\
  n& &  & 1 &  & &  &  \\
  n+1& &  & \sqrt{2}m & 1 & & &  \\
  n+2 & & & m^2 & \sqrt{2}m & 1 & \\
  & &  &  & & & \ddots  & \\
  &&&&&&&1\\ 
\end{block} 
\end{blockarray}\hspace{5pt}, 
\end{equation*}
where all unmarked off-diagonal entries are $0$.  

\section{Compatibility of pinnings}
\label{sec:compatibility}

We record here some key properties of the pinnings of $\operatorname{Sp}_{2n}$ and $\operatorname{SO}_{2n+1}$.
These properties describe how their pinnings are compatible with the standard pinnings of $GL_{2n}$ and $GL_{2n+1}$ under the embeddings $\operatorname{Sp}_{2n}\hookrightarrow GL_{2n}$ and $\operatorname{SO}_{2n+1}\hookrightarrow GL_{2n+1}$.
To avoid repeated arguments in this and subsequent sections, we shall often use notations as in the following general setup.

\begin{setup}\label{setup}
Let $G$ be a connected, reductive, $\RR$-split linear algebraic group with a fixed pinning $(T,B_+,B_-,\{x_i\}_{i\in I}, \{y_i\}_{i\in I})$ with simple roots $I$ in the root system $\Phi$.
Let $\iota: G \hookrightarrow GL_N$ be an embedding, and fix a function $\psi: I \to \{\text{nonempty sequences in $[N-1]$}\}$.  We write $\psi$ also for the function $\{\text{sequences in $I$}\} \to \{\text{sequences in $[N-1]$}\}$ defined by
\[
(i_1, \dotsc, i_\ell) \mapsto \big(\text{the concatenation of }\psi(i_1),\dotsc, \psi(i_\ell)\big).
\]
\end{setup}

For $\iota: \operatorname{Sp}_{2n} \hookrightarrow GL_{2n}$, we define $\psi(i) = (i,2n-i)$ for $i\in [n-1]$, and $\psi(n) = n$.
For $\iota: \operatorname{SO}_{2n+1} \hookrightarrow GL_{2n+1}$, we define $\psi(i) = (i,2n+1-i)$ for $i\in [n-1]$, and $\psi(n) = (n,n+1,n)$.
The following lemma is verified straightforwardly from the explicit descriptions of the pinnings of $\operatorname{Sp}_{2n}$ and $\operatorname{SO}_{2n+1}$.

\begin{lem}\label{lem:pinning1}
The pinning $(T^C, B_+^C, B_-^C, \{x_i^C\}_{i\in [n]}, \{y_i^C\}_{i\in [n]})$ of $\operatorname{Sp}_{2n}$ relates to the standard pinning of $SL_{2n}$ by $\chi_i^C(t) = \chi_i^A(t)\chi_{2n-i}^A(t)$ if $i\in [n-1]$ and $\chi_n^C(t) = \chi_n^A(t)$, and moreover,
\[
y_i^C(m) = \begin{cases}
y_i^A(m)y_{2n-i}^A(m) & \text{if $i\in [n-1]$}\\
y_n^A(m) & \text{if $i = n$}
\end{cases} \quad\text{and}\quad 
\dot s_i^C = \begin{cases}
\dot s_i^A \dot s_{2n-i}^A & \text{if $i\in [n-1]$}\\
\dot s_n^A & \text{if $i = n$}.
\end{cases}
\]

The pinning $(T^B, B_+^B, B_-^B, \{x_i^B\}_{i\in [n]}, \{y_i^B\}_{i\in [n]})$ of $\operatorname{SO}_{2n+1}$ relates to the standard pinning of $SL_{2n+1}$ by $\chi_i^B(t) = \chi_i^A(t)\chi_{2n+1-i}^A(t)$ if $i\in [n-1]$ and $\chi_n^B(t) = \chi_n^A(t^2)\chi_{n+1}^A(t^2)$, and moreover,
\[
y_i^B(m) = \begin{cases}
y_i^A(m)y_{2n+1-i}^A(m) & \text{if $i\in [n-1]$}\\
y_{n}^A(\frac{m}{\sqrt2})y_{n+1}^A(\sqrt{2}m) y_{n}^A(\frac{m}{\sqrt2})& \text{if $i = n$}
\end{cases} \quad\text{and}\quad 
\dot s_i^B = \begin{cases}
\dot s_i^A \dot s_{2n+1-i}^A & \text{if $i\in [n-1]$}\\
\dot s_{n}^A\dot s_{n+1}^A\dot s_{n}^A & \text{if $i = n$}.
\end{cases}
\]
\end{lem}

Note also that in $SL_{2n+1}$, we have
\[\textstyle
y_{n}^A(\frac{m}{\sqrt2})y_{n+1}^A(\sqrt2m) y_{n}^A(\frac{m}{\sqrt2})  =   y_{n+1}^A(\frac{m}{\sqrt2})y_{n}^A(\sqrt2m) y_{n+1}^A(\frac{m}{\sqrt2}) 
\quad\text{and}\quad
\dot s_{n}^A\dot s_{n+1}^A\dot s_{n}^A = \dot s_{n+1}^A\dot s_n^A\dot s_{n+1}^A.
\]

In the lemma above, the notation $y_i^A$ denoted a map into either $GL_{2n}$ or $GL_{2n+1}$, depending on context, and likewise for $\chi_i^A$ and $\dot s_i^A$.  We will continue this abuse of notation, as we trust that this ambiguity will cause no confusion.

\medskip
In what follows, 
we describe in detail the implications of Lemma~\ref{lem:pinning1} for partial flag varieties in Section~\ref{sect:partialflag}, and for Bruhat orders of Weyl groups in Section~\ref{sect:Bruhat}.

\subsection{Partial flag varieties}\label{sect:partialflag}

Let $G$ be as in the Setup~\ref{setup}.
For a subset $J\subseteq I$, let $W_J = \langle s_i : i\in J\rangle$ be the corresponding parabolic subgroup of $W$, and let $P_J$ be the corresponding parabolic subgroup of $G$ containing $B_+$ (so $P_\emptyset = B_+$).
When $I = [n]$, given $J\subseteq [n]$, we often denote $K := [n]\setminus J$.

\medskip
In type A with the standard pinning of $GL_n$, for $k\in [n-1]$ the quotient $GL_n/P_{[n-1]\setminus\{k\}}$ is identified with the Grassmannian $\operatorname{Gr}_{k;n} = \{L \text{ a $k$-dimensional subspace of } \RR^n\}$ by taking the column span of first $k$ columns.
Let $\binom{[n]}{k}$ denote the set of $k$-subsets of $[n]$.
The \emph{Pl\"ucker embedding} $\operatorname{Gr}_{k;n} \hookrightarrow \PP(\bigwedge^k \RR^n) \simeq \PP(\RR^{\binom{[n]}{k}})$ is given by
\[
\operatorname{Gr}_{k;n}\ni L \mapsto (\Delta_S(A))_{S\in\binom{[n]}{k}}
\]
where $A$ is a $n\times k$ matrix $A$ whose column span is $L$, and $\Delta_S(A)$ is the maximal minor of $A$ corresponding to the rows labelled by $S$.
In this case, we call the sequence $(\Delta_S(A))_{S\in\binom{[n]}{k}}$ the \emph{Pl\"ucker coordinates} of $L$. Pl\"ucker coordinates are well defined projective coordinates, that is, they are well-defined up to a global nonzero scalar multiple.

\medskip
We now describe the partial flag varieties of $\operatorname{Sp}_{2n}$ and $\operatorname{SO}_{2n+1}$.
We prepare with the following lemma.
For a subspace $L$ of a finite dimensional vector space $V$ with a fixed nondegenerate symmetric or alternating bilinear form $\mathcal B(\cdot,\cdot)$, denote by $L^\perp := \{v\in V : \mathcal B(v,\ell) = 0 \text{ for all }\ell\in L\}$.
Note that $L$ is isotropic if and only if $L\subseteq L^\perp$, and that $(L^\perp)^\perp = L$.
We say that $L$ is \emph{coisotropic} if $L^\perp \subseteq L$.

\begin{lem}\label{lem:pinning2}
Let $E$ be an anti-diagonal $n\times n$ matrix with alternating $\pm1$, defining a symmetric or alternating bilinear form, depending on $n$.  
For a matrix $M$, let $L_i$ denote the span of its first $i$ columns.
Then, for any matrix $M$ satisfying $M^tEM = E$, we have $L_{n-i} = L_i^\perp$.  Moreover, for any $L\subseteq \RR^n$, the set of Pl\"ucker coordinates of $L$ and that of $L^\perp$ are equal (up to a global nonzero scalar).
\end{lem}

\begin{proof}
To see $L_{n-i} = L_i^\perp$, one notes that the top-left $(n-i)\times i$ submatrix of $E$ is the zero matrix, which implies that $L_{n-i}$ pairs trivially under $E$ with $L_i$.  That $\dim L_{n-i} + \dim L_i = n$ then implies that $L_{n-i} = L_i^\perp$.

For the statement about the Pl\"ucker coordinates, let us first recall some multilinear algebra.
Let $0\to L \to V \to M \to 0$ be a short exact sequence of vector spaces with $\dim V = n$ and $\dim L = k$.  Multiplying via the wedge product on the left defines the map $\bigwedge^k V \to \operatorname{Hom}(\bigwedge^{n-k}V, \bigwedge^n V)$.
The image of $\bigwedge^k L$ under this map, which is a line, is equal to the image of the natural map $\operatorname{Hom}(\bigwedge^{n-k}M, \bigwedge^n V) \to \operatorname{Hom}(\bigwedge^{n-k}V, \bigwedge^n V)$.
Choosing a basis $(\be_1, \dotsc, \be_n)$ of $V$, and the isomorphism $\bigwedge^n V \simeq \RR$ where $\be_1 \wedge \dotsm \wedge \be_n \mapsto 1$, the map $\bigwedge^k V \to \bigwedge^{n-k}V^\vee$ is then given by
\[
\be_I \mapsto \operatorname{sign}(I,[n]\setminus I) \be_{[n]\setminus I}^\vee,
\]
where we denote $\be_I = \be_{i_1} \wedge \dotsm \wedge \be_{i_k}$ for $I = \{i_1 < \dotsb < i_k\}$ (and similarly for the dual vectors $\be_I^\vee$), and $\operatorname{sign}(I,[n]\setminus I)$ is the sign of the permutation $(I, [n]\setminus I)$ of $[n]$ in which both $I$ and $[n]\setminus I$ are ordered in the increasing order.

Now, let $\varphi$ be the isomorphism $V \overset\sim\to V^\vee$ induced by the pairing $E$.
Under this isomorphism, the subspace $M^\vee \subseteq V^\vee \simeq V$ is equal to $L^\perp$.
Hence, the desired result about Pl\"ucker coordinates of $L$ and $L^\perp$ follows from the following claim: denoting by $\overline J := \{n+1-j : j\in J\}$ for $J\subseteq [n]$, one has that the composition $\bigwedge^k V \to \bigwedge^{n-k}V^\vee \to \bigwedge^{n-k}V$ is given by either
\[
\be_I \mapsto \be_{\overline{[n]\setminus I}} \quad\text{for all $I\in \binom{[n]}{k}$\qquad or \qquad} \be_I \mapsto -\be_{\overline{[n]\setminus I}} \quad \text{for all $I\in \binom{[n]}{k}$}.
\]
To show this claim, we first note that an explicit description of $\varphi$ is given by
\[
\varphi(\be_i) = (-1)^{n-i}\be_{n+1-i}^\vee \quad \text{for all $i\in[n]$}.
\]
The composition $\bigwedge^k V \to \bigwedge^{n-k}V^\vee \to \bigwedge^{n-k}V$ is thus given by
\[
\be_I \mapsto \operatorname{sign}(I,[n]\setminus I)  \cdot (-1)^{\#\{j \in [n]\setminus I : n-j\text{ odd}\}} \cdot \be_{\overline{[n]\setminus I}}.
\]
The sign of $\operatorname{sign}(I,[n]\setminus I)  \cdot (-1)^{\#\{j \in [n]\setminus I : n-j\text{ odd}\}}$ is independent of the $k$-subset $I$, since replacing $i\in I$ with $j\in [n]\setminus I$ either changes the signs in both factors or leaves both unchanged.
\end{proof}

For a subset $K  = \{k_1 < \dotsb < k_j\} \subseteq [n]$ and an integer $m$, denote by $m-K$ the set $\{m-k_j < \dotsb < m-k_1\}$ with $0$ omitted if it occurs.
We have the following descriptions for the partial flag varieties of $\operatorname{Sp}_{2n}$ and $\operatorname{SO}_{2n+1}$.
Let $\operatorname{SpFl}_{K;2n}$ and $\operatorname{SOFl}_{K;2n+1}$ (and $\operatorname{SpFl}_{K;2n}^{\Delta>0}$ and $\operatorname{SOFl}_{K;2n+1}^{\Delta>0}$) be as defined in the introduction.
Denote $J = [n]\setminus K$.

\begin{cor}\label{cor:pinning3}
For a subset $K \subseteq [n]$, we have
\begin{align*}
\operatorname{Sp}_{2n}/P_J^C &= \{L_\bullet \in \operatorname{Fl}_{K\cup (2n-K);2n} : L_{i} = L_j^\perp \text{ if $\dim L_i + \dim L_j = 2n$}\} \simeq  \operatorname{SpFl}_{K;2n}, \quad\text{and}\\
\operatorname{SO}_{2n+1}/P_J^B &= \{L_\bullet \in \operatorname{Fl}_{K\cup (2n+1-K);2n+1} :  L_{i} = L_j^\perp  \text{ if $\dim L_i + \dim L_j = 2n+1$}\} \simeq  \operatorname{SOFl}_{K;2n+1}.
\end{align*}
Moreover, under the last isomorphisms, we have $(\operatorname{Sp}_{2n}/P_J^C) \cap \operatorname{Fl}_{K\cup (2n-K);2n}^{\Delta >0} \simeq \operatorname{SpFl}_{K;2n}^{\Delta>0}$ and $(\operatorname{SO}_{2n+1}/P_J^B) \cap \operatorname{Fl}_{K\cup (2n+1-K);2n+1}^{\Delta >0} \simeq \operatorname{SOFl}_{K;2n+1}^{\Delta>0}$, and similarly with $\Delta \geq 0$ in place of $\Delta >0$.
\end{cor}

\begin{proof}
It follows from Lemma~\ref{lem:pinning1} that
\[
P_J^C = \operatorname{Sp}_{2n} \cap P_{[2n-1]\setminus (K \cup (2n-K))}^A \quad\text{and}\quad P_J^B = \operatorname{SO}_{2n+1} \cap P_{[2n]\setminus (K \cup (2n+1-K))}^A.
\] 
In particular, the left-hand-sides of the equations in the corollary are subsets of the right-hand-sides by Lemma~\ref{lem:pinning2}.
For the other inclusion, let us prove the type C case; the type B case is similar.  Given such an $L_\bullet$ in the right-hand-side $\operatorname{Fl}_{K\cup (2n-K);2n}$, we need to construct $A\in \operatorname{Sp}_{2n}$ such that $L_\bullet = L_\bullet(A)$. Since every isotropic subspace is contained in an isotropic space of dimension $n$ (by Witt's theorem), and since subspaces of isotropic subspaces are isotropic, it suffices to do the case when $K = [n]$.
Pick a basis $(v_1, \dotsc, v_n, v_{n+1}, \dotsc, v_{2n})$ of $\RR^{2n}$ such that $v_i \in L_i\setminus L_{i-1}$ for all $i = 1, \dotsc, n$ (where $L_0 := 0$).
Let $(v_1^*, \dotsc, v_{2n}^*)$ be the dual basis with respect to the nondegenerate alternating form $\langle \cdot, \cdot \rangle$ given by $E$, where $\langle v_i, v_i^* \rangle = -\langle v_i^*, v_i\rangle = 1$.
Then, because $v_{n+1-j}^* \in L_{n-j}^\perp \setminus L^\perp_{n-(j-1)}$ for all $j = 1, \dotsc, n$, and since $L_1 \subsetneq \dotsb \subsetneq L_n = L_n^\perp \subsetneq L_{n-1}^\perp \subsetneq \cdots \subsetneq L_1^\perp$ is a complete flag, we find that $(v_1, \dotsc, v_n, v_n^*, \dotsc, v_1^*)$ is a basis of $\RR^{2n}$.
The matrix $A$ consisting of this basis as its columns is the desired $A\in \operatorname{Sp}_{2n}$.
Lastly, the statements about the Pl\"ucker positive or nonnegative parts follow from Lemma~\ref{lem:pinning2}
\end{proof}

\subsection{Bruhat orders}\label{sect:Bruhat}

Let $W$ be a Weyl group with simple reflections $\{s_i: i\in I\}$.
We use a bold letter for an expression (i.e.\ a word) in the simple reflections, whose un-bolded letter denotes the element in $W$ obtained by multiplying the simple reflections in that expression. We will say an expression $\bw=s_{i_1}\cdots s_{i_l}$ for $w\in W$ has length $\ell(\bw)=l$. The expression is \emph{reduced} if it is an expression for $w$ of minimal length. We write $\ell(w)$ for the length of a reduced expression for $w$.
Denote by $w_0\in W$ the element of longest length.

For a reduced expression $\bv=s_{i_1}\cdots s_{i_p}$ of an element $v\in W$, a \emph{subexpression} $\bu$ of $\bv$ is a choice of either $1$ or $s_{i_j}$ for each $j\in[p]$. We will record this by writing $\bu$ as a string whose $j^{\textnormal{th}}$ entry is either $1$ or $s_{i_j}$.
We may interpret the subexpression $\bu$ as an expression for some $u\in W$ by ignoring the $1$s.
The \emph{Bruhat order} $<$ on $W$ is defined by $u<v$ if there is a subexpression for $u$ in some, equivalently any, expression for $v$.  
For a general background on Weyl groups and Bruhat orders, we refer to \cite{BjornerBrenti} or \cite{Humphreys}.

\medskip
The Weyl groups of $\operatorname{Sp}_{2n}$ and $\operatorname{SO}_{2n+1}$ are isomorphic and known as \emph{signed permutation groups}.
We refer the reader to \cite[Chapter 8.1]{BjornerBrenti} for relevant background on signed permutation groups, and record the key facts that we need here. 
In terms of the embeddings 
$\operatorname{Sp}_{2n}\hookrightarrow GL_{2n}$ and $\operatorname{SO}_{2n+1}\hookrightarrow GL_{2n+1}$,
Lemma~\ref{lem:pinning1} yields the following two realizations of the signed permutation group as a subgroup of a type A Weyl group.

\begin{rem}\label{rem:weyl}
For $\operatorname{Sp}_{2n}$, relabel $[2n]$ as $[n,\overline n] :=\{1, 2, \dotsc, n, \overline n, \dotsc, \overline 2, \overline 1\}$, where we view the ``bar'' as an involution $\overline{\overline{i}}=i$.  Then, Lemma~\ref{lem:pinning1} identifies $s_i^C$ with $(i \ i+1)(\overline i \ \overline{i+1})$ for $i  \in [n-1]$, and $s_n^C$ with $(n \ \overline n)$.
In particular, we may identify $W^C$ as the subgroup of $\mathfrak S_{2n}$ consisting of permutations $\sigma$ of $[n,\overline n]$ such that $\sigma(\overline{i}) = \overline{\sigma(i)}$ for all $i \in [n]$.
Similarly for $\operatorname{SO}_{2n+1}$, relabel $[2n+1]$ as $[n,0,\overline n] :=\{1,2, \dotsc, n, 0, \overline n, \dotsc, \overline 2, \overline 1\}$.  Then, we have the same description for the $s_i^B$, and $W^B$ is the subgroup of $\mathfrak S_{2n+1}$ consisting of permutations $\sigma$ of $[n,0,\overline n]$ such that $\sigma(0) = 0$ and $\sigma(\overline i) = \overline{\sigma(i)}$ for all $i\in [n]$.
\end{rem}

Using the descriptions of $W^C$ and $W^B$ in the remark, 
we now record explicit descriptions for lengths of elements in terms of \emph{inversions}. 

\begin{defn}
Let us linearly order $1<2<\cdots<n<0<\overline{n}<\cdots<\overline{2}<\overline{1}$.
For $w\in W^A\simeq \mathfrak{S}_n$, a pair $(i,j)\in [n]\times [n]$ is an \emph{inversion} of $w$ if $i<j$ and $w(j)<w(i)$ . 
 For $w\in W^C\subset \mathfrak{S}_{2n}$ (resp.\ $W^B\subset \mathfrak{S}_{2n+1}$), a pair $(i,j) \in [n]\times [n,\overline n]$ is an \emph{inversion} if $i<j$ and $w(j)<w(i)$, or equivalently, $w(\overline{i})<w(\overline{j})$. 
\end{defn}

Definitions of inversions in other Weyl groups can be found in \cite{BjornerBrenti}.

\begin{prop}
Let $W=W^A, W^B,$ or $W^C$.  For $w\in W$, its length $\ell(w)$ is the number of inversions of $w$.
\end{prop}

We use inversions to record some useful facts about reduced expressions and subexpressions, and how $\psi$ acts on them.
We work with two pairs $(\iota,\psi)$ here, namely, $\iota:Sp_{2n}\hookrightarrow GL_{2n}$ and $\iota:SO_{2n+1}\hookrightarrow GL_{2n+1}$, with the corresponding maps $\psi$ defined below Setup~\ref{setup}. We will slightly abuse notation by referring to both pairs by the same symbols, trusting that it is clear from context which we mean.
For the rest of this section, $\Phi$ stands for $C$ or $B$ whenever it appears.
For an expression $\bv = s_{i_1}^\Phi \dotsm s_{i_\ell}^\Phi$ in $W^\Phi$, let us write $\psi(\bv)$ for the expression $\mathbf s_{\psi(i_1, \dotsc, i_\ell)}$ in $W^A$.

\begin{rem}
For the proof of Theorem~\ref{thm:main}, the only property we will need is the following: there exists a reduced word $\mathbf{w_0}$ for $w_0^\Phi \in W^\Phi$ such that $\psi(\mathbf{w_0})$ is a reduced word for $w_0^A \in W^A$. Corollary~\ref{cor:preserveslongestword} here states that for \emph{any} reduced word $\mathbf{w_0}$ for $w_0^\Phi$, the expression $\psi(\mathbf{w_0})$ is a reduced word for $w_0^A$.
The reader seeking a minimal path to Theorem~\ref{thm:main} may skip the rest of this section by verifying this property directly for a choice of reduced word for $w_0^\Phi$.

\end{rem}
 
\begin{prop}\label{prop:psipreservesreduced}
    If $\bv$ is a reduced expression for some $v\in W^\Phi$, then $\psi(\bv)$ is a reduced expression as well.
\end{prop}

\begin{proof}
Let $\mathfrak{S}^\pm_n$ be the signed symmetric group.
Recall that $s_i^C$ generates a subgroup of $\mathfrak{S}_{2n}$ isomorphic to $\mathfrak{S}^\pm_n$, and $s_i^B$ generates a subgroup of $\mathfrak{S}_{2n+1}$ isomorphic to $\mathfrak{S}^\pm_n$.
We proceed by induction on length in $\mathfrak{S}^\pm_n$. When $\ell(\bv)=0$, there is nothing to check.
    
    We first consider reduced expressions $\bv=\bs_\bi^C$ for elements $v$ of $W^C\cong \mathfrak{S}^\pm_n \subset \mathfrak S_{2n}$. If $\bv$, $\psi(\bv)$, and $\bv s_i^{C}$ are all reduced expressions, we want to show that $\psi(\bv s_i^{C})$ is reduced as well. By reducedness, $s_i^{C}$ multiplied to the right of $\bv$ introduces an inversion to $v$. By definition of an inversion in $\mathfrak{S}^\pm_n$, and from the explicit description of $\mathfrak{S}^\pm_n$ and $\mathfrak{S}_{2n}$ given in \Cref{rem:weyl}, multiplication by $\psi(s_i^{C})$ adds either one or two inversions to $\psi(v)$ depending on whether $i=n$ or $i\neq n$. Since $\psi(s_i^{C})$ is a product of one or two simple transpositions, respectively, $\psi(\bv s_i^{C})$ is reduced. 

    We next consider reduced expressions $\bv=\bs_\bi^B$ for elements $v$ of $W^B\cong \mathfrak{S}^\pm_n \subset \mathfrak S_{2n+1}$. If $\bv$, $\psi(\bv)$, and $\bv s_i^{B}$ are all reduced expressions, we want to show that $\psi(\bv s_i^{B})$ is reduced as well. By reducedness, $s_i^{B}$ multiplied to the right of $\bv$ introduces an inversion to $v$. By definition of an inversion in $\mathfrak{S}^\pm_n$, and from the explicit description of $\mathfrak{S}^\pm_n$ and $\mathfrak{S}_{2n+1}$ given in \Cref{rem:weyl}, multiplication by $\psi(s_i^{B})$ adds either three or two inversions to $\psi(v)$ depending on whether $i=n$ or $i\neq n$. Since $\psi(s_i^{B})$ is a product of three or two simple transpositions, respectively, $\psi(\bv s_i^{B})$ is reduced. 
\end{proof}

\begin{prop}\label{prop:reduced}
    Let $\bv$ be a reduced expression for some $v\in W^\Phi$. Then, $\bv s_i^{\Phi}$ is reduced if and only if $\psi(\bv)\psi(s_i^\Phi)$ is reduced. 
\end{prop}

\begin{proof}
    Suppose $\bv s_i^\Phi$ is reduced. Then, since $\psi(\bv s_i^\Phi)=\psi(\bv)\psi(s_i^\Phi)$, the latter is reduced by \Cref{prop:psipreservesreduced}. If $\bv s_i^\Phi$ is not reduced, then there is a reduced subexpression $\bu$ contained in it such that $vs_i^\Phi=u$. By \Cref{prop:psipreservesreduced}, $\psi(\bu)$ is reduced. Both $\psi(\bu)$ and $\psi(\bv)\psi(s_i^\Phi)$ are expressions for the same Weyl group element. Also, $\psi(\bu)$ is a strict subexpression of $\psi(\bv)\psi(s_i^\Phi)$. Thus, $\psi(\bv)\psi(s_i^\Phi)$ is not reduced.
\end{proof}

\begin{cor}\label{cor:reducedinAfromC}
    For $i<n$ and $\bv=\bs_\bi^C$ an expression for some $v\in W^C$, either both or neither of $\psi(\bv)s_i^{A}$ and $\psi(\bv)s_i^{A}s_{2n-i}^A$ are reduced.
\end{cor}

\begin{proof}
    This follows directly from \Cref{prop:reduced}, using our explicit description of $s_i^C$ in \Cref{rem:weyl}.
\end{proof}

\begin{cor}\label{cor:reducedinAfromB}
    For $i<n$ and $\bv=\bs_\bi^B$ an expression for some $v\in W^B$, either both or neither of $\psi(\bv)s_i^{A}$ and $\psi(\bv)s_i^{A}s_{2n+1-i}^A$ are reduced. Moreover, either all or none of $\psi(\bv)s_{n+1}^A$, $\psi(\bv)s_{n+1}^As_n^A$, $\psi(\bv)s_{n+1}^As_n^As_{n+1}$, $\psi(\bv)s_{n}^A$, and $\psi(\bv)s_{n}^As_{n+1}^A$ are reduced.
\end{cor}

\begin{proof}
    This follows directly from \Cref{prop:reduced}, using our explicit description of $s_i^B$ in \Cref{rem:weyl} and the fact that $s_{n+1}^As_n^As_{n+1}^A=s_{n}^As_{n+1}^As_{n}^A$. 
\end{proof}

\begin{cor}\label{cor:preserveslongestword}
    If $\bv$ is a reduced expression for $w_0^\Phi\in W^\Phi$, then $\psi(\bv)$ is a reduced expression for $w_0^A \in W^A$. 
\end{cor}
\begin{proof}
    By \Cref{prop:psipreservesreduced}, $\psi(\bv)$ is reduced. Moreover, by \Cref{prop:reduced} and \Cref{cor:reducedinAfromC} or \Cref{cor:reducedinAfromB}, depending on whether $\Phi=C$ or $\Phi=B$, multiplication by any $s_i^A$ causes $\psi(\bv)$ to no longer be reduced.
\end{proof}

\section{Lusztig positivity and the proof of (3)\texorpdfstring{$\implies$}{=>}(1)\texorpdfstring{$\implies$}{=>}(2)}\label{sect:Luspositivity}

\subsection{Lusztig's total positivity}\label{subsect:Luspositivity}
Let $G$ be a connected, reductive, $\RR$-split linear algebraic group with a fixed pinning $(T,B_+,B_-,\{x_i\}_{i\in I}, \{y_i\}_{i\in I})$.
We recall Lusztig's definition of total positivity for $G$.

\begin{defn}\label{defn:LusztigPos}
For a sequence $\mathbf i$ in $I$ such that $\mathbf s_{\mathbf i}$ is a reduced expression for the longest element $w_0\in W$,
define $U_-^{>0}$ (resp.\ $U_-^{\geq 0}$) to be the image $\mathbf y_{\mathbf i}(\RR^{|\mathbf i|}_{>0})$ (resp.\ $\mathbf y_{\mathbf i}(\RR^{|\mathbf i|}_{\geq0})$),
and similarly define $U_+^{>0}$ and $U_+^{\geq 0}$ in terms of $\mathbf{x}_{\mathbf i}$.
Define $T^{>0}$ to be the subgroup of the $\RR$-split torus $T$ generated by the elements $\chi(t)$ for $t\in \RR_{>0}$ and $\chi: \RR^* \to T$ a cocharacter of $T$.
Define the \emph{positive} (resp.\ \emph{nonnegative}) \emph{part} of $G$ to be
\[
G^{>0} := U_-^{>0}T^{>0}U_+^{>0} \quad\text{(resp.\ $G^{\geq 0} := U_-^{\geq0}T^{> 0}U_+^{\geq0}$)}.
\]
\end{defn}

The sets $\mathbf y_{\mathbf i}(\RR^{|\mathbf i|}_{>0})$ and $\mathbf y_{\mathbf i}(\RR^{|\mathbf i|}_{\geq0})$ depend only on the element $\mathbf s_{\mathbf i}\in W$ as long as $\mathbf{s}_\mathbf i$ is a reduced expression \cite{lusztig1994positivity}.  In particular, the space $G^{>0}$ as defined is independent of the choice of the reduced word for $w_0$.
When $G = GL_n$ with the standard pinning, it is a classical result \cite{Cryer73, Cryer76}
that $GL_n^{>0}$ (resp.\ $GL_n^{\geq 0}$) as defined here is the space of invertible matrices with all positive (resp.\ nonnegative) minors.
For a detailed survey of totally positive matrices, see \cite{FallatJohnsonTNN}.

\medskip
For a parabolic subgroup $P \subset G$ containing $B_+$, let $\pi: G \to G/P$ be the projection map to the partial flag variety $G/P$.  
For a subset $S\subseteq G/P$, we denote by $\overline S$ its closure with respect to the Euclidean topology on $(G/P)(\RR)$.

\begin{defn} Define the \emph{positive} (resp.\ \emph{nonnegative}) \emph{part} of the partial flag $G/P$ to be
\[
(G/P)^{>0} := \pi(G^{>0}) \quad\text{(resp.\ $(G/P)^{\geq0} := \overline{\pi(G^{>0})}$)}.
\]
\end{defn}

We caution that although $G^{\geq 0}$ is the closure of $G^{>0}$ \cite{lusztig1994positivity}, the image $\pi(G^{\geq 0})$ may be strictly contained in $(G/B_+)^{\geq 0}$, since $\pi: G \to G/{B_+}$ may not be proper.
However, note that the projection map $G/{B_+} \to G/P$ is proper, and hence $(G/P)^{\geq 0}$ is the image of $(G/B_+)^{\geq 0}$.

\subsection{Proof of (1)\texorpdfstring{$\implies$}{=>}(2)}
We first show a statement in the general setting of Setup~\ref{setup}, where $\iota: G \hookrightarrow GL_N$ is an embedding and $\psi$ maps sequences in $I$ to sequences in $[N-1]$. Let $G$ have root system $\Phi$.

\begin{defn}\label{condition1}
We say that $(\iota,\psi)$ has \emph{property ($\dagger$)} if the following are satisfied.
\begin{enumerate}
\item[($\dagger1$)] For every $i\in I$, we have $\dot s_i^\Phi = \dot{\mathbf s}_{\psi(i)}^A$, and we have $y_i^\Phi(a) = \mathbf y_{\psi(i)}^A(f_1(a), \dotsc, f_{|\psi(i)|}(a))$ for some sequence $(f_1, \dotsc, f_{|\psi(i)|})$ of differentiable functions $f_j: \RR \to \RR$ such that $f_j(\RR_{>0}) \subseteq \RR_{>0}$,
and similarly for $x_i^\Phi$ and $\chi_i^\Phi$. 
\item[($\dagger2$)] For some sequence $\mathbf i$ in $I$ such that ${\mathbf s}^{\Phi}_\mathbf i$ is a reduced word for the longest element $w_0\in W^\Phi$, the word ${\mathbf s}_{\psi(\mathbf i)}^A$ is a reduced word for the longest element of $W^A$. 
\end{enumerate}
\end{defn}

Note that each $f_j$ in \hyperref[condition1]{($\dagger1$)} satisfies $f_j(0) = 0$ because $y_i^\Phi: \RR \to G$ is a homomorphism. 
Note also that the property \hyperref[condition1]{($\dagger1$)} implies that $B_+^\Phi \subseteq B_+^A$.  In particular, if $(\iota, \psi)$ satisfies \hyperref[condition1]{($\dagger$)} and $P$ is a parabolic subgroup of $GL_N$ containing $B_+^A$, then $P \cap G$ is a parabolic subgroup of $G$ containing $B_+^\Phi$.

\begin{prop}\label{prop:dagger}
Suppose $(\iota, \psi)$ has property \hyperref[condition1]{($\dagger$)}.
Then, we have the following.
\begin{enumerate}[label = (\alph*)]
\item $G^{>0} \subseteq GL_N^{>0}$.
\item For $J\subseteq [N-1]$, let $P^A_J \supseteq B_+^A$ be the parabolic subgroup of $GL_N$, so that $GL_N/P^A_J = \operatorname{Fl}_{K;N}$ where $K = [N-1]\setminus J$. Let $P = P^A_J \cap G$.  Then, we have that
\[
(G/P) \cap \operatorname{Fl}_{K;N}^{\Delta \geq 0} 
= \overline{\Big( (G/P) \cap \operatorname{Fl}_{K;N}^{\Delta >0} \Big)}.
\]
\end{enumerate}
\end{prop}

\begin{proof}
The property \hyperref[condition1]{($\dagger$)} implies that $(U_-^\Phi)^{>0}\subseteq(U_-^A)^{>0}$, $(T^\Phi)^{>0}\subseteq(T^A)^{>0}$, and $(U_+^\Phi)^{>0}\subseteq(U_+^A)^{>0}$, from which the statement (a) follows.  
For the statement (b), let $p$ be a point in $(G/P) \cap \operatorname{Fl}_{K;N}^{\Delta \geq 0}$.
We exhibit $p$ as a limit of points in $(G/P) \cap \operatorname{Fl}_{K;N}^{\Delta > 0}$, as follows.  Let $K = \{k_1, \ldots, k_j\}$.

For an $N\times N$ matrix $M$, let $L_\bullet(M)$ denote the flag of subspaces whose $i$-th subspace is the span of the first $k_i$ columns of $M$.
By definition, there is an $N \times N$ matrix $M$ in $\iota(G)$ such that $p= L_\bullet(M)$, and such that for every $1\leq i \leq j$, the first $k_i$ columns of $M$ have all nonnegative maximal minors, at least one of which is positive by invertibility of $M$. 
By the Cauchy--Binet formula, for any $A\in GL_N^{>0}$ and $1\leq i\leq j$, the first $k_i$ columns of $AM$ have all positive maximal minors, i.e.\ $L_\bullet(AM) \in \operatorname{Fl}_{K;N}^{\Delta > 0}$.  Thus, for any curve $A(t): \RR_{>0} \to G^{>0} \subseteq GL_N^{>0}$ such that $\lim_{t\to 0} A(t)$ equals the identity, the curve $L_\bullet(A(t)M)$ lies in $(G/P) \cap \operatorname{Fl}_{K;N}^{\Delta >0}$ and limits to $p$ as $t\to 0$.
Such a curve exists since $G^{\geq 0}$ is the closure of $G^{>0}$, or explicitly, one may take $A(t) = \mathbf y_{\mathbf i}^{\Phi}(t,\dotsc, t) \boldsymbol \chi_{\mathbf i}^{\Phi}(1, \dotsc, 1) \mathbf x_{\mathbf i}^{\Phi}(t,\dotsc, t)$ where $\mathbf s_{\mathbf i} = w_0^\Phi$.
\end{proof}

Applying this proposition to $\operatorname{Sp}_{2n}$ and $\operatorname{SO}_{2n+1}$ yields Theorem~\ref{thm:main} \ref{mainthm:pos}$\implies$\ref{mainthm:nonneg}, as follows.

\begin{proof}[Proof of Theorem~\ref{thm:main} \ref{mainthm:pos}$\implies$\ref{mainthm:nonneg}]
Since by definition $(G/P)^{\geq 0}$ is the closure of $(G/P)^{>0}$, the implication \ref{mainthm:pos}$\implies$\ref{mainthm:nonneg} would follow from Corollary~\ref{cor:pinning3} once we show that the Pl\"ucker nonnegative part is the closure of the Pl\"ucker positive part.
To show this, by Proposition~\ref{prop:dagger}(b) it suffices to verify that the embeddings $\operatorname{Sp}_{2n} \hookrightarrow GL_{2n}$ and $\operatorname{SO}_{2n+1} \hookrightarrow GL_{2n+1}$, along with our choice of pinnings, satisfy the property \hyperref[condition1]{($\dagger$)}.
That these embeddings satisfy \hyperref[condition1]{($\dagger$)} follows from Lemma~\ref{lem:pinning1} and Corollary~\ref{cor:preserveslongestword}. 
\end{proof}

\subsection{Proof of (3)\texorpdfstring{$\implies$}{=>}(1)}
We again work first in the general setting as stated in Setup~\ref{setup} involving an embedding $\iota: G \hookrightarrow GL_N$ and a map $\psi$.
We prepare by recalling the following fact.

\begin{lem}[{\cite[Proposition 5.2, Theorem 11.3]{marshrietsch2004}}]\label{lem:positiveparam}
Suppose $\mathbf{s}_{\mathbf i}$ is a reduced expression for the longest element $w_0 \in W$ of length $\ell=\ell(w_0)$, and let $\mathring{\mathbf y}_{\mathbf i}: (\RR^*)^\ell \to G/B_+$ be the composition of the restriction of the map $\mathbf y_{\mathbf i}: \RR^\ell \to G$ to the torus $(\RR^*)^\ell$ with the projection map $G\to G/B_+$.  The map $\mathring{\mathbf y}_{\mathbf i}$ is an isomorphism onto a Zariski open subset of $G$.  In particular, the map $\mathring{\mathbf y}_{\mathbf i}$ has a dense image in $G/B_+$, and induces a bijection $(\RR_{>0})^\ell \overset\sim\to (G/B_+)^{>0}$.
\end{lem}

\begin{prop}\label{prop:strongdagger}
Suppose the pair $(\iota, \psi)$ satisfies the property \hyperref[condition1]{($\dagger$)}, and satisfies the following strengthening of \hyperref[condition1]{($\dagger1$)}:
\begin{itemize}\label{condition1strong}
\item[($\dagger1'$)]
For each $i\in I$, the functions $f_j: \RR \to \RR$ appearing in \hyperref[condition1]{($\dagger1$)} further satisfy the property that:
\begin{itemize}
    \item[$\bullet$] $f_j(\RR_{< 0}) \subseteq \RR_{< 0}$ (so $f_j(\RR^*) \subset \RR^*$), and
    \item[$\bullet$] the image of the map $\RR^* \to (\RR^*)^{|\psi(i)|}$ given by $a \mapsto (f_1(a), \dotsc, f_{|\psi(i)|}(a))$ is closed (in Euclidean topology).
\end{itemize}
\end{itemize}
Then, one has
\[
(G/B_+)^{>0} = (G/B_+) \cap \operatorname{Fl}_{[N-1];N}^{\Delta >0}.
\]
\end{prop}

\begin{proof}
By Theorem~\ref{thm:BK} (also \cite[Theorem 4.11]{boretsky2023totally}), we may prove the statement with $\operatorname{Fl}_{[N-1];N}^{\Delta >0}$ replaced by $\operatorname{Fl}_{[N-1];N}^{>0}$.
Fix a sequence $\mathbf i$ such that $\mathbf{s}_\mathbf{i}$ is a reduced expression for the longest element of $W^\Phi$, as in \hyperref[condition1]{($\dagger2$)}, and denote $U = \operatorname{image}(\mathring{\mathbf y}_{\mathbf i}^\Phi)$ and $V = \operatorname{image}(\mathring{\mathbf y}_{\psi(\mathbf i)}^A)$.
The strengthening \hyperref[condition1strong]{($\dagger1'$)} implies that
$U$ is a closed subset of $V$ under the embedding $G/B_+ \hookrightarrow \operatorname{Fl}_{[N-1];N}$ induced by $\iota: G \hookrightarrow GL_N$.
As $U$ is dense in $G/B_+$ by Lemma~\ref{lem:positiveparam}, and as $G/B_+ \hookrightarrow \operatorname{Fl}_{[N-1];N}$ is a closed embedding, we find that $U$ is dense and closed in $(G/B_+) \cap V$, and hence $U = (G/B_+) \cap V$.
The desired equality for the Lusztig positive part now follows from Lemma~\ref{lem:positiveparam}, since \hyperref[condition1]{($\dagger1$)} and \hyperref[condition1strong]{($\dagger1'$)} imply that $f_j^{-1}(\RR_{>0}) = \RR_{>0}$.
\end{proof}

We now prove Theorem~\ref{thm:main} \ref{mainthm:consec}$\implies$\ref{mainthm:pos}.

\begin{proof}[Proof of Theorem~\ref{thm:main} \ref{mainthm:consec}$\implies$\ref{mainthm:pos}]
Let $n\geq 2$.
As the embeddings $\operatorname{Sp}_{2n} \hookrightarrow GL_{2n}$ and $\operatorname{SO}_{2n+1} \hookrightarrow GL_{2n+1}$ satisfy \hyperref[condition1]{($\dagger$)}, Definition~\ref{defn:LusztigPos} and Proposition~\ref{prop:dagger}(a) together imply that $\operatorname{SpFl}_{K ; 2n}^{>0} \subseteq \operatorname{SpFl}_{K; 2n}^{\Delta>0}$ and $\operatorname{SOFl}_{K ; 2n+1}^{>0} \subseteq \operatorname{SOFl}_{K; 2n+1}^{\Delta>0}$ for any $K\subseteq [n]$.
It remains to show the reverse inclusions when $K = \{k, k +1, \dotsc, n\}$ for some $k\in [n]$.

We first reduce to the case $K = [n]$ as follows.
We show this reduction for the $\operatorname{Sp}_{2n}$ case; the case of $\operatorname{SO}_{2n+1}$ is similar.
By Corollary~\ref{cor:pinning3}, a point in $\operatorname{SpFl}_{K; 2n}^{\Delta>0}$ is a point $L_\bullet$ in $(\operatorname{Sp}_{2n}/P_K^C)\cap \operatorname{Fl}_{K\cup (2n-K); 2n}^{\Delta >0}$.
Since $K\cup (2n-K)$ consists of consecutive integers by our assumption on $K$, Theorem~\ref{thm:BK} implies $ \operatorname{Fl}_{K\cup (2n-K); 2n}^{\Delta >0} =  \operatorname{Fl}_{K\cup (2n-K); 2n}^{>0}$.  Since by definition $\operatorname{Fl}_{K\cup (2n-K); 2n}^{>0}$ is the projection of $\operatorname{Fl}_{[2n-1]; 2n}^{>0}$, we may extend the flag $L_\bullet$ to a flag $\widetilde L_\bullet$ in $\operatorname{Fl}_{[2n-1]; 2n}^{>0} = \operatorname{Fl}_{[2n-1]; 2n}^{\Delta>0}$.
Because subspaces of isotropic subspaces are isotropic, the projection of $\widetilde L_\bullet$ to $\operatorname{Fl}_{[n];2n}$ is a point in $\operatorname{SpFl}_{[n];2n}^{\Delta >0}$.
In particular, by Lemma~\ref{lem:pinning2}, we may choose $\widetilde L_\bullet$ such that $\widetilde L_\bullet \in (\operatorname{Sp}_{2n}/B_+) \cap \operatorname{Fl}_{[2n-1]; 2n}^{\Delta>0}$.
Hence, if Lusztig positivity and Pl\"ucker positivity agrees for the case of $K = [n]$, then $\widetilde L_\bullet \in (\operatorname{Sp}_{2n}/B)^{>0}$ so that its projection $L_\bullet$ is Lusztig positive also.

Lastly, the case of $K = [n]$ follows from Proposition~\ref{prop:strongdagger} and Corollary~\ref{cor:pinning3} once we show that the embeddings $\operatorname{Sp}_{2n} \hookrightarrow GL_{2n}$ and $\operatorname{SO}_{2n+1} \hookrightarrow GL_{2n+1}$ satisfy the property \hyperref[condition1]{$(\dagger$)} and the strengthening \hyperref[condition1strong]{($\dagger1'$)}.
The property \hyperref[condition1]{$(\dagger$)} was already verified in the proof of \ref{mainthm:pos}$\implies$\ref{mainthm:nonneg}, and \hyperref[condition1strong]{($\dagger1'$)} follows from Lemma~\ref{lem:pinning1}, which displays that the $f_j$ are linear functions with positive coefficients.
\end{proof}

\begin{rem}
For situations where the functions $f_j$ appearing in \hyperref[condition1]{($\dagger1$)} are not simply linear functions, one condition that implies \hyperref[condition1strong]{($\dagger1'$)} is as follows.
\begin{itemize}
\item[($\dagger1''$)]\label{condition1strongstrong}
For each $i\in I$, the functions $f_j: \RR \to \RR$ appearing in \hyperref[condition1]{($\dagger1$)} further satisfy the property that:
    \begin{itemize}
        \item[$\bullet$] $f_j(\RR_{< 0}) \subseteq \RR_{< 0}$ (so $f_j(\RR^*) \subset \RR^*$),
        \item[$\bullet$] $\lim\limits_{a\to+\infty}f_j(a)=+\infty$ for at least one $j\in\psi(i)$, and
        \item[$\bullet$] $\lim\limits_{a\to-\infty}f_{j}(a)=-\infty$ for at least one $j\in\psi(i)$.
    \end{itemize}
\end{itemize}
One can check that \hyperref[condition1strongstrong]{($\dagger1''$)} implies \hyperref[condition1strong]{($\dagger1'$)} through a simple application of the Bolzano--Weierstrass Theorem. The second and third bullet points cannot be dropped, as shown by counterexamples such as $f_j(a)=\arctan(a)$, $\forall j\in\psi(i)$.
\end{rem}

\begin{rem}\label{rem:karpman}
    Karpman \cite[Proposition 6.1]{karpman2018} showed that the statements (1) and (2) in \Cref{thm:main} hold for Lagrangian Grassmannians $\operatorname{SpFl}_{n;2n}$. However, her method does not extend to other symplectic Grassmannians $\operatorname{SpFl}_{k;2n}$, as indicated by our main theorem. Specifically, when $k\neq n$, the poset of type C projected Richardson varieties in $\operatorname{SpFl}_{k;2n}$ is incompatible with the poset of positroid varieties in $\operatorname{Gr}_{k;2n}$, which implies that \cite[Corollary 3.3]{karpman2018} and \cite[Proposition 3.4]{karpman2018} cannot be generalized to $k\neq n$. Instead, we adopted a new approach by leveraging \Cref{thm:BK} from \cite{boretsky2023totally} and \cite{blochkarp2023}.
\end{rem}

\section{Examples, non-examples, and the proof of (2)\texorpdfstring{$\implies$}{=>}(3)}\label{sect:constructions}

\subsection{Proof of (2)\texorpdfstring{$\implies$}{=>}(3)}

Let $G$ be either $\operatorname{Sp}_{2n}$ or $\operatorname{SO}_{2n+1}$.
As we have shown $(3) \implies (1) \implies (2)$, in particular for $K = [n]$, we have that $(G/B_+)^{\geq 0} = (G/B_+)^{\Delta \geq 0}$.
We now provide examples to demonstrate the contrapositive of \ref{mainthm:nonneg}$\implies$\ref{mainthm:consec}, as follows.
For each relevant $K\subseteq [n]$ and $J = [n]\setminus K$, we will find a Pl\"ucker nonnegative point in $G/P_J$ that does not extend to a Pl\"ucker nonnegative point in $G/B_+$.
Such a point cannot be in the Lusztig nonnegative part $(G/P_J)^{\geq 0}$, since $(G/P_J)^{\geq 0}$ is the projection of $(G/B_+)^{\geq 0} = (G/B_+)^{\Delta \geq 0}$.

\subsubsection{Type C examples}
Our examples for $\operatorname{Sp}_{2n}$ will be constructed from the following observation.

\begin{lem}\label{lem:smallcases}
For $n\geq 2$, define by $L_{1;2n}$ the subspace
$L_{1;2n} = \operatorname{span}(\be_1 + \be_{2n}) \subset \RR^{2n}$.
Then, $L_{1;2n}$ is a point in $\mathrm{SpFl}_{1;2n}^{\Delta \geq 0}$ that does not extend to a point in $\mathrm{SpFl}_{1,2;2n}^{\Delta \geq 0}$.
\end{lem}

\begin{proof}
Since $L_{1;2n}$ is 1-dimensional, it is isotropic with respect to the alternating form $\omega$ defined in the introduction.  It is also manifestly Pl\"ucker nonnegative.
Now, suppose for a contradiction that $L_{1;2n}$ extends to a point in $\mathrm{SpFl}_{1,2;2n}^{\Delta \geq 0}$.  That is, suppose we have a $2$-dimensional isotropic subspace $L_2$ containing $L_{1;2n}$ such that $(L_{1;2n}, L_2) \in \mathrm{SpFl}_{1,2;2n}^{\Delta \geq 0}$.  Such a subspace $L_2$ is the row-span of a matrix of the form
\[
\begin{bmatrix}
1 & 0 &  \dotsc & 0 & 1\\
0 & x_2 & \dotsc & x_{2n-1} & x_{2n}
\end{bmatrix}.
\]
In order for $2\times 2$ minors to be nonnegative, we find $x_2 = \dotsb = x_{2n-1} = 0$, but then the isotropicity of the row-span of the matrix demands $x_{2n} = 0$.  This contradicts that $L_2$ is 2-dimensional.
\end{proof}

Let us denote by $[L_{1;2n}]$ the vector
\[
[L_{1;2n}] = \begin{bmatrix} 1 & 0 &  \dotsc & 0 & 1\end{bmatrix},
\]
whose row-span is $L_{1;2n}$.

\medskip
We now demonstrate how to construct a point in $\mathrm{SpFl}_{K;2n}^{\Delta \geq 0}\setminus  \mathrm{SpFl}_{K;2n}^{\geq 0}$ for all $n\geq 2$ and $K$ not of the form $\{k, k+1, \dotsc, n\}$ for some $1 \leq k \leq n$. 
Fix such a subset $K$, and denote $g$ to be the integer satisfying $g\notin K$ and $\{g+1, g+2, \dotsc, n\} \subset K$, i.e.\ ``the first gap from the right.''
We have three cases:
(i) $g = n$, (ii) $g = n-1$, and (iii) $g \leq n-2$.
In each case, we will produce a point $L_\bullet \in \mathrm{SpFl}_{K;2n}^{\Delta \geq 0}$ that does not extend to a point in $\mathrm{SpFl}_{[n];2n}^{\Delta \geq 0}$.
For all cases, denote $f = \max\{i \mid i\in K \text{ and } i < g\}$.  Note that $f+1 \notin K$.

\smallskip
\noindent \textbf{Case (i).} By assumption $f\leq n-1$.  Consider the $f\times 2n$ matrix
\[
M = \begin{bmatrix}
I_{f-1} & 0 & 0 \\
0 & [L_{1;2(n-f+1)}] & 0
\end{bmatrix},
\]
where $I_{f-1}$ is the $(f-1)\times (f-1)$ identity matrix.  Observe that the rows are linearly independent, and that for all $1\leq i \leq k$, the matrix formed by the first $i$ rows of $M$ has isotropic row-span and has all nonnegative Pl\"ucker coordinates; that is, the matrix defines a point $L_\bullet \in \operatorname{SpFl}_{K;2n}^{\Delta \geq 0}$.
However, we claim that it cannot be extended to a point in $\operatorname{SpFl}_{K\cup \{f+1\};2n}^{\Delta \geq 0}$.
Suppose there is such an extension, say by adding a row vector $\mathbf v$ to the matrix.  Then, we may assume that $v_1 = \dotsc = v_{f-1} = 0$ (since row-reduction does not change the row-span), and furthermore $v_{2n} = \dotsc = v_{2n + 1 - (f-1)} = 0$ since $\omega(\be_i, \mathbf v) = 0$ for all $1\leq i \leq f-1$.
Since the form $\omega$ restricted to $\operatorname{span}(\be_f, \dotsc, \be_{2n-f+1})$ can be identified with the usual alternating form on $\RR^{2(n-f+1)}$ (up to a global sign), we have from Lemma~\ref{lem:smallcases} then that $\mathbf v = 0$.

\smallskip
\noindent \textbf{Case (ii).}
By assumption $f \leq n-2$.  Let $\ell = n-f-2$.  Consider the $n\times 2n$ matrix
\[
M = \begin{bmatrix}
I_{f-1} & \vline & 0 & \vline & 0& \vline & \quad\qquad0\quad\qquad \\
\hline
0 & \vline & 0 & \vline & \begin{matrix} 1 & 0 & 0 & 0 & 0 & 1 \end{matrix}& \vline & \quad\qquad0\quad\qquad\\
\hline
0 & \vline & (-1)^{\ell}I_{\ell}  & \vline& 0 & \vline & \quad\qquad0\quad\qquad \\
\hline
0 & \vline & 0  & \vline & \begin{matrix} 0 & 1 & 0 & 0 & 0 & 0\\ 0 & 0 & 1 & 0 & 0 & 0\end{matrix} & \vline & \quad\qquad0\quad\qquad 
\end{bmatrix}.
\]
Observe that all rows are linearly independent, and that for all $i\in \{1, 2,\dotsc, f\} \cup \{n\}$, the matrix formed by the first $i$ rows of $M$ has isotropic row-span and has all nonnegative maximal minors.
That is, the matrix defines a point $(L_i)_{i\in K} \in \operatorname{SpFl}_{K;2n}^{\Delta \geq 0}$.
Suppose for a contradiction that it can be extended to a point $(L_1 \subset \dotsb \subset L_n) \in \operatorname{SpFl}_{[n];2n}^{\Delta \geq 0}$.
We recall that intersecting by a coordinate subspace spanned by standard basis vectors of consecutive indices preserves Pl\"ucker nonnegativity \cite[Lemma 2.12]{blochkarp2023}.
We thus let $L' = \operatorname{span}(\be_{n-2}, \be_{n-1}, \dotsc, \be_{n+3})$, and let $((L_1 \cap L') \subseteq \dotsb \subseteq (L_n\cap L'))$ be the intersection flag (with repetitions removed), considered as a point in $\operatorname{SpFl}_{K';6}^{\Delta \geq 0}$ for some $K'$. 
Since $\dim (L_i\cap L') - \dim (L_{i-1} \cap L')\leq 1$ for all $i$, we find that $K'$ must consist of consecutive integers.  Since by construction $\{1,3\} \subset K'$, we have $K' = [3]$.  This in particular implies that we have an extension of $L_{1;6}$ to a point in $\operatorname{SpFl}_{[3];6}^{\Delta \geq 0}$, which contradicts Lemma~\ref{lem:smallcases}.

\smallskip
\noindent \textbf{Case (iii).} We use the construction from Bloch and Karp.  Namely, consider the $n\times 2n$ matrix
\[
M = \begin{bmatrix}
0 & I_{f-1} & 0 &0 \\
B & 0 & 0 &0\\
0 & 0 & I_{n-f-3} &0\\
C & 0 & 0 & 0
\end{bmatrix}
\]
where
\[B = (-1)^{f-1}\begin{bmatrix}
1 & 0 & 0 & 1\\
0 & 1 & 0 & 0\\
0 & 0 & 1 & 0
\end{bmatrix}
\quad\text{and}\quad  C = (-1)^{n-f-3}\begin{bmatrix} 0 &0 &0& 1\end{bmatrix}.
\]
Since every row of $M$ is contained in $\operatorname{span}(\be_1, \dotsc, \be_n)$, row-spans are isotropic.
Moreover, for every $i \neq f+1$, the matrix formed by the first $i$ rows of $M$ has all nonnegative maximal minors.  In particular, the matrix defines a point $L_\bullet$ in $\operatorname{SpFl}_{K;2n}^{\Delta \geq 0}$.  Bloch and Karp showed that, considered as a point in $\operatorname{Fl}_{K;n}^{\Delta \geq 0}$, the flag $L_\bullet$ does not admit an extension to a point in $\operatorname{Fl}_{[n];n}^{\Delta \geq 0}$ \cite[Proof of Theorem 1.1 (ii)$\implies$(iii)]{blochkarp2023}.

\subsubsection{Type B examples}
Since \ref{mainthm:nonneg}$\implies$\ref{mainthm:consec} fails for $B_2$ (see Remark~\ref{rem:B2}), we assume $n\geq 3$ throughout.
Our examples will be constructed from the following observation.

\begin{lem}\label{lem:Bsmallcases}
For $n\geq 3$, let the subspace $L_{1;B_n} \subset \RR^{2n+1}$ be defined by
\[
L_{1;B_n} := 
\begin{cases}
\operatorname{span}(\be_1 + \sqrt{2}\be_{n+1}+\be_{2n+1}) & \text{if $n$ is odd}\\
\operatorname{span}(\be_2 + \sqrt{2}\be_{n+1}+\be_{2n}) & \text{if $n$ is even}.
\end{cases}
\]
Then, $L_{1;B_n}$ is a point in $\mathrm{SOFl}_{1;2n+1}^{\Delta \geq 0}$ that does not extend to a point in $\mathrm{SOFl}_{[n];2n+1}^{\Delta \geq 0}$.
\end{lem}

\begin{proof}
In both cases (odd and even $n$), one verifies the isotropicity with respect to the symmetric form $Q$ defined in the introduction by $2\cdot1\cdot1 - \sqrt{2}^2 = 0$.
Pl\"ucker nonnegativity is clear.
We show more strongly that $L_{1;2n}$ does not extend to a point $L_\bullet$ in $\mathrm{SOFl}_{1,2,3;2n+1}^{\Delta \geq 0}$.
Suppose for a contradiction otherwise.
We treat the two cases, odd and even $n$, separately.

When $n$ is odd, the 2-dimensional subspace $L_2$ in the flag $L_\bullet$ is the row-span of a matrix of the form 
\[
\begin{bmatrix}
1 & 0 &  \dotsc &0 & \sqrt{2}& 0 & \dotsc & 0 & 1\\
0 & x_2 &  \dotsc & x_n & x_{n+1}& x_{n+2} & \dotsc & x_{2n} & x_{2n+1}
\end{bmatrix}.
\]
In order for $2\times 2$ minors to be nonnegative, we find $x_2 = \dotsb = x_{n} = x_{n+2} = \dotsb = x_{2n} = 0$, but then the isotropicity of the second row of the matrix demands $x_{n+1} = 0$, after which the isotropicity of the two rows with each other demands $x_{2n+1} = 0$.  This contradicts that $L_2$ is 2-dimensional.

When $n$ is even, the 2-dimensional subspace $L_2$ in the flag $L_\bullet$ is the row-span of a matrix of the form
\[
\begin{bmatrix}
0 & 1 & 0 &  \dotsc &0 & \sqrt{2}& 0 & \dotsc & 0 & 1 & 0\\
-a & 0 & 0 &  \dotsc &0 & b& 0 & \dotsc & 0 & c & d
\end{bmatrix}.
\]
Here, several entries in the second row have been forced to be 0 by row-reduction or the nonnegativity of the 2$\times$2 minors.  Because the isotropicity of the two rows implies $2ad = b^2$ and $c = \sqrt{2}b$, we find that either $a$ or $d$ must be nonzero, since otherwise $a = b= c= d= 0$.  Without loss of generality, let us set $a = 1$, since the argument that follows is similar for the case of $d =1$ by symmetry.  Now, the 3-dimensional subspace $L_3$ in the flag $L_\bullet$ is the row-span of a matrix of the form
\[
\begin{bmatrix}
0 & 1 & 0 &  \dotsc &0 & \sqrt{2}& 0 & \dotsc & 0 & 1 & 0\\
-1 & 0 & 0 &  \dotsc &0 & b& 0 & \dotsc & 0 & c & d\\
0 & 0 & x_3 &  \dotsc &x_n & x_{n+1}& x_{n+2} & \dotsc & x_{2n-1} & x_{2n} & x_{2n+1}\\
\end{bmatrix}.
\]
For each $i \in \{3, \dotsc, n\} \cup \{n+2, \dotsc, 2n-1\}$, we have that $x_i= P_{12i} = -P_{1i(2n)}$, so that Pl\"ucker nonnegativity implies $x_i = 0$.  Then, the isotropicity of the third row implies $x_{n+1} = 0$, after which the isotropicity of the third row with the first and the second row implies $x_{2n} = 0$ and $x_{2n+1} = 0$, respectively.  This contradicts that $L_3$ is 3-dimensional.
\end{proof}

Let us write $[L_{1;B_n}]$ for the nonnegative matrix whose row-span is $L_{1;B_n}$.
We make one more observation in preparation.

\begin{lem}\label{lem:Bcornercase}
The 2-dimensional subspace defined as the row-span of
\[
\begin{bmatrix}
1 & 2 & 2 & 2 & 2 & 1 & 0\\
0 & 1 & 2 & 2 & 2 & 2 & 1
\end{bmatrix}
\]
is a point in $\mathrm{SOFl}_{1,2;7}^{\Delta\geq 0}$ that does not extend to a point in $\mathrm{SOFl}_{[3];7}^{\Delta\geq 0}$.
\end{lem}

\begin{proof}
Suppose it extends, so that the 3-dimensional subspace in the extension is the row-span of a matrix of the form
\[
\begin{bmatrix}
1 & 1 & 0 & 0 & 0 & -1 & -1\\
0 & 1 & 2 & 2 & 2 & 2 & 1\\
0 & 0 & a & b & c & d & e
\end{bmatrix}.
\]
The nonnegativity of the Pl\"ucker coordinates $P_{134}$ and $P_{145}$ implies $a\leq b \leq c$, whereas the nonnegativity of $P_{346}$ and $P_{456}$ implies $a \geq b \geq c$, so that $a= b = c$.
Then, the isotropicity of the bottom row implies $2ac =b^2$, which implies $a = b= c = 0$.  Lastly, the isotropicity of the bottom row with the first and the second row implies $d = e$ and $d = 0$, respectively.
\end{proof}

We now demonstrate how to construct a point in $\mathrm{SOFl}_{J;2n+1}^{\Delta \geq 0}\setminus  \mathrm{SOFl}_{J;2n+1}^{\geq 0}$ for all $n\geq 3$ and $K$ not of the form $\{k, k+1, \dotsc, n\}$ for some $1 \leq k \leq n$.
Given the Lemmas~\ref{lem:Bsmallcases} and \ref{lem:Bcornercase}, the arguments are nearly identical to the type C case, so we will omit details.

Fix such a subset $K$, and denote $g$ to be the integer satisfying $g\notin K$ and $\{g+1, g+2, \dotsc, n\} \subset K$, i.e.\ ``the first gap from the right.''
We have three cases:
(i) $g = n$, (ii) $g = n-1$, and (iii) $g \leq n-2$.
For all cases, denote $f = \max\{i \mid i\in K \text{ and } i < g\}$.  Note that $f+1 \notin K$.

\smallskip
\noindent \textbf{Case (i).}
By assumption $f\leq n-1$.
We further divide into two cases.  When $f = n-1$, consider the $k\times (2n+1)$ matrix
\[
M = \begin{bmatrix}
I_{f-2} & 0 & 0 \\
0 & A & 0
\end{bmatrix} \quad\text{where}\quad A= \begin{bmatrix}
1 & 2 & 2 & 2 & 2 & 1 & 0\\
0 & 1 & 2 & 2 & 2 & 2 & 1
\end{bmatrix}.
\]
When $f \leq n-2$, so that $n-f+1 \geq 3$, consider the $f\times (2n+1)$ matrix
\[
M = \begin{bmatrix}
I_{f-1} & 0 & 0 \\
0 & [L_{1;B_{n-f+1}}] & 0
\end{bmatrix}.
\]
In both cases, the appropriate row-spans define a point in $\mathrm{SOFl}_{K;2n+1}^{\Delta\geq 0}$ that does not extend to a point in $\mathrm{SOFl}_{[n];2n+1}^{\Delta\geq 0}$.

\smallskip
\noindent \textbf{Case (ii).}
By assumption $f \leq n-2$.  Let $\ell = n-f-2$.  The $n\times (2n+1)$ matrix
\[
M = \begin{bmatrix}
I_{f-1} & \vline & 0 & \vline & 0& \vline & \quad0\quad  &\vline & 0\quad \\
\hline
0 & \vline & 0 & \vline & \begin{matrix} 1 & 0 & 0 & \sqrt{2} & 0 & 0 & 1 \end{matrix}& \vline & \quad0\quad&\vline & 0\quad \\
\hline
0 & \vline & (-1)^{\ell}I_{\ell}  & \vline& 0 & \vline & \quad0\quad&\vline & 0\quad \\
\hline
0 & \vline & 0  & \vline & \begin{matrix} 0 & 1 & 0 & 0 & 0 & 0 & 0\\ 0 & 0 & 1 & 0 & 0 & 0 & 0\end{matrix} & \vline & \quad0\quad &\vline & 0\quad 
\end{bmatrix}
\]
provides the desired example.

\smallskip
\noindent \textbf{Case (iii).} In this case, the construction from Bloch and Karp, identical to the one recalled here in the type C case, provides the desired example.

\begin{rem}\label{rem:B2}
Let us treat the $B_2$ case.  We claim that every point in $\operatorname{SOFl}_{1;5}^{\Delta> 0}$ extends to a point in $\operatorname{SOFl}_{1,2;5}^{\Delta> 0}$.
Let $L$ be the row-span of a positive matrix $[a \ b \ c \ d \ e]$.  For it to be isotropic, we have $2ae+c^2 = 2bd$.
We may assume $a = 1$, and thus $e = bd - \frac{c^2}{2}$.
Consider then the matrix
\[
\begin{bmatrix}
1 & b & c & d & bd-\frac{c^2}{2}\\
0 & 1 & 2x & 2x^2 & 2bx^2+d-{2}cx
\end{bmatrix}.
\]
One verifies the isotropicity, and notes that every minor is a polynomial in $x$ whose leading term is positive.
Hence, for a sufficiently large $x>0$, the matrix represents a point in $\operatorname{SOFl}_{1,2;5}^{\Delta> 0}$.
\end{rem}

\section{Type D counterexamples}\label{sect:typeD}
In this section, we will discuss counterexamples that arise in type D when applying the methods used above in type B and C.
First, we show that the proof strategy used for types B and C does not apply in type D, by showing that there is no type D pinning satisfying \hyperref[condition1]{($\dagger$)}.
A test for Lusztig positivity in terms of Pl\"ucker coordinates in type D will be shown using a different method in an forthcoming work with the second author.
Second, we show that replacing Pl\"ucker positivity by ``Pfaffian positivity'' fails to detect Lusztig positivity.

\subsection{Type D pinning}

We take the type D Lie group $\operatorname{SO}_{2n}$ to be defined as the linear subgroup of $\operatorname{SL}_{2n}$ that preserves a non-degenerate symmetric bilinear form $Q$ on $\mathbb{R}^{2n}$.  Specifically, we have:
\[\operatorname{SO}_{2n}:=\{A\in \operatorname{SL}_{2n}(\mathbb{R})|A^tEA=E\},\]
where $E$ is the symmetric matrix associated with $Q$. In this section, we will show that there is no pinning of $\operatorname{SO}_{2n}$ with nice enough properties to be treated the same way as type B and type C in this paper.

\begin{prop}\label{prop:typeDpinning}
    There does not exist a choice of $E$ (which determines the embedding $\iota:\operatorname{SO}_{2n}\hookrightarrow\operatorname{SL}_{2n}$), a pinning $(T^D, B_+^D,B_-^D,\{x_i^D\},\{y_i^D\})$ of $\operatorname{SO}_{2n}$, and a map $\psi$ satisfying \hyperref[condition1]{($\dagger1$)} in \Cref{condition1}.
 
\end{prop}

We first prove the following lemma.

\begin{lem}\label{lem:typeDpinning}
    Given a symmetric matrix $E$, which determines an embedding $\iota:\operatorname{SO}_{2n}\hookrightarrow\operatorname{SL}_{2n}$. If the maximal torus $T^D$ of $\operatorname{SO}_{2n}$ satisfies $\iota(T^D)\subseteq T^A$, then $E$ must be a monomial matrix.
\end{lem}
\begin{proof}
    It is easier to work with the Lie algebra
    \[\mathfrak{so}_{2n}=\{A\in \mathfrak{sl}_{2n}\mid AE+EA^t=0\}.\]
    Let $t=\text{diag}(t_1,t_2,\dots,t_{2n})$ be a diagonal matrix in $\mathfrak{sl}_{2n}$. Then $t\in \mathfrak{so}_{2n}$ implies $tE+Et=0$, or $(t_i+t_j)E(i,j)=0$ for any $i,j$ (here $E(i,j)$ is the $(i,j)^\text{th}$ entry of $E$) Therefore, $t_i+t_j=0$ if $E(i,j)\neq 0$. We create a graph $G$ on $2n$ vertices, where we connect $i,j$ if the entry $E(i,j)\neq 0$. Note we add a self-loop to vertex $i$ in this graph if $E(i,i)\neq 0$. Note that $\dim(T^D)$ is at most the number of connected components of $G$ since all $t_i$'s in that connected component are the same up to a sign. Also, since $E$ is not degenerate, all vertices in $G$ have degree at least one. As a result, any connected component of size one must be a self loop, at some vertex $i$. This forces $t_i=0$. Assume $G$ has $a$ connected components with size $1$ and $b$ connected components with size $\geq 2$. Then $a+2b\leq 2n$, and $n=\dim(T^D)\leq b$. Both inequalities hold simultaneously if and only if $a=0$, $b=n$ and all connected components have size $2$. This implies $E$ is a monomial matrix that represents a permutation $w\in S_{2n}$, where $w$ is an involution with no fixed points.
\end{proof}

\begin{proof}[Proof of Proposition~\ref{prop:typeDpinning}]
    Let $e_i^A\in \mathfrak{sl}_{2n}$ and $e_i^D\in \mathfrak{so}_{2n}$ be the Chevalley generators of the Lie algebras. If \hyperref[condition1]{($\dagger1$)} holds, where $y_i^D(a)=\mathbf{y}^A_{\psi(i)}(f_1(a), \dots, f_{|\psi(i)|}(a))=\prod_{j\in \psi(i)}y_{i_j}^A(f_j(a))$, then:
    \[e_i^D=\lim_{a\to 0}\frac{y_i^D(a)-I}{a}=\lim_{a\to 0}\frac{\prod_{j\in \psi(i)}y_{i_j}^A(f_j(a))-I}{a}=\lim_{a\to 0}\frac{\prod_{j\in \psi(i)}(I+f_j(a)\cdot e^A_{i_j})-I}{a}=\sum_{j\in I}f'_j(0)e_{i_j}^A.\]
    Therefore, $e_i^D$ is a sub-diagonal matrix, meaning all non-zero entries lies on the diagonal directly above the main diagonal. Let the support of $e_i^D$, denoted as $\text{supp}(e_i^D)$, be the set of entries in $e_i^D$ that are non-zero. The contradiction arises as follows:
    \begin{enumerate}
        \item Note that $e_i^DE$ is a skew-symmetric matrix since $e_i^D\in \mathfrak{so}_{2n}$. If $|\text{supp}(e_i^D)|=1$, then $e_i^DE$ would have exactly one non-zero entry,
        given that $E$ is a monomial matrix by Lemma~\ref{lem:typeDpinning}. This is impossible for a skew-symmetric matrix, as it would violate the property of having an even number of non-zero entries. Therefore, we conclude that $|\text{supp}(e_i^D)|\geq 2$.
        \item For any $i\neq j$, $\text{supp}(e_i^D)\cap\text{supp}(e_j^D)=\emptyset$. To see this, let $t=\text{diag}(t_1,\dots,t_{2n})$ be a generic element in $T^D$. If they did share entry $(k,k+1)$, then since $e_i^D$ and $e_j^D$ are both eigenvectors of $T^D$, they both must belong to the $\lambda=(t_k-t_{k+1})$ eigenspace in order to satisfy $[t,e_i^D]=\lambda e_i^D$ in position $(k,k+1)$ and the analogous equation for $e_j^D$. This contradicts the fact that $e_i^D$ and $e_j^D$ belong to different $T^D$-eigenspaces.

        \item Now, all $\text{supp}(e_i^D)$ for $i\in [n]$ are disjoint and each has size at least two. However, there are only $2n-1$ entries on the sub-diagonal. This is a contradiction.\qedhere
    \end{enumerate}
    
    \end{proof}

\begin{rem}
    We conjecture that Proposition~\ref{prop:typeDpinning} should hold for any embedding $\iota: \operatorname{SO}_{2n}\hookrightarrow\operatorname{SL}_N$.
\end{rem}

\subsection{Pfaffian positivity}
We have so far compared Pl\"ucker positivity with Lusztig positivity. Pl\"ucker coordinates may be thought of as arising from the action of $G$ on exterior powers of the standard representation of $G$. Types B and D also have spin representations, not arising from tensor powers of the standard representation, which give coordinates resembling matrix Pfaffians. One might wonder whether these coordinates also may be used for positivity tests. 
In this section, we will show that these Pfaffian coordinates fail to detect Lusztig positivity for an orthogonal Grassmannian of type D$_4$. This also gives a counterexample to \cite[Chapter 4, Proposition 5.1]{Rietsch} and the note at the end of \cite{lusztig1998positivity}, which (as a special case) asserted that the canonical basis for a fundamental representation detects positivity for its associated Grassmannian.

\subsubsection{Pinning for $\operatorname{SO}_{2n}$}

Set $E_0=\left[\begin{smallmatrix}
         & &(-1)^n \\ &\iddots &\\ -1& &
    \end{smallmatrix}\right]$. We take $\mathfrak{so}_{2n}$ to be the $2n\times 2n$ matrices $A$ so that $A^tE+EA=0$, where
\[ E = \begin{bmatrix}
    & E_0^t \\ E_0 &
\end{bmatrix}=\left[\begin{smallmatrix}
    &&&&&&&-1\\
    &&&&&&1&\\
    &&&&&\iddots&&\\
    &&&&(-1)^n&&&\\
    &&&(-1)^n&&&&\\
    &&\iddots&&&&&\\
    &1&&&&&&\\
    -1&&&&&&&
\end{smallmatrix}\right]. \]
Then $\operatorname{O}_{2n}$ is the invertible $2n\times 2n$ matrices $A$ so that $A^tEA=E$, and $\operatorname{SO}_{2n}$ is its subgroup of determinant 1 matrices. We take the maximal torus to be $\mathrm{diag}(t_1,\ldots,t_n,t^{-1}_n,\ldots,t^{-1}_1)$. 

\begin{rem}
    We pick the embedding of $\operatorname{SO}_{2n}$ into $\operatorname{GL}_{2n}$ so that the matrix $E$ defining $\operatorname{SO}_{2n}$ is antidiagonal, to maintain consistency with the other groups considered in this paper. One could instead use  $E = \begin{bmatrix}
    0 & I_n \\ I_n & 0
\end{bmatrix},$
which would make the
canonical basis coordinates for the orthogonal Grassmannian (equivalently, generalized minors for the spin representation) coincide with matrix Pfaffians on the nose. The two embeddings are related via conjugation by $\begin{bmatrix}
    I_n & \\
    & E_0
\end{bmatrix}$. See \cite[Remarks 5.2 and 5.5]{GalPyl} for more discussion.
\end{rem}

The simple roots are $\{\be_1-\be_2,\ldots, \be_{n-1}-\be_{n}, \be_{n-1}+\be_n\}$. We write the corresponding fundamental weights as $\varpi_1,\ldots, \varpi_n$, and denote the irreducible representation of highest weight $\lambda$ by $V_{\lambda}$. For $i\in [n-1]$, the map $\phi_i^{D_{n}}$ is given by
\begin{equation*}
    \phi_i^{D_{n}}\left(\begin{pmatrix}
        a&b\\
        c&d
    \end{pmatrix}\right)=\begin{blockarray}{cccccccccc}
& &  & i & i+1 & & 2n-i& 2n-i+1 & &\\
\begin{block}{c(ccccccccc)}
  & 1 &  &  & & & & & &\\
  & & \ddots &  & &  & & & & \\
  i& &  & a & b &  & & & & \\
  i+1& &  &c  & d &  &  & & & \\
  & &  &  &  & \ddots  & & & &  \\
  2n-i&&&&&&a&b&&\\
  2n-i+1&&&&&&c&d&&\\
 &&&&&&&&\ddots&\\
 &&&&&&&&&1\\
\end{block} 
\end{blockarray}\hspace{5pt}, 
\end{equation*}
where all unmarked off-diagonal entries are $0$.
For $i=n$, the map $\phi_n^{D_n}$ is given by

\begin{equation*}
    \phi_n^{D_{n}}\left(\begin{pmatrix}
        a&b\\
        c&d
    \end{pmatrix}\right)=\begin{blockarray}{cccccccccc}
& &  & n-1 & n & n+1& n+2 & & &\\
\begin{block}{c(ccccccccc)}
  & 1 &  &  &  & & & & &\\
  & & \ddots &   &  & & & & &\\
  n-1& &  & a &    & b& & && \\
  n& &  &  & a &   & b& & &\\
  n+1&&&c&&d&&&&\\
  n+2&&&&c&&d&&&\\
 &&&&&&&&\ddots&\\
 &&&&&&&&&1\\
\end{block} 
\end{blockarray}\hspace{5pt}, 
\end{equation*}

Set $\overline{i} = 2n-i+1$. The Weyl group $\mathfrak{D}_n$ acts via permutations of $\{1,\ldots,n, \overline{n},\ldots, \overline{1}\}$ so that $\sigma(\overline{i})=\overline{\sigma(i)}$ and so that an even number of un-barred elements are sent to barred elements. 

\subsubsection{Spin representations}
Let $Q(\cdot,\cdot)$ be the symmetric form on $\CC^{2n}$ determined by $E$ and let $\{\be_i\}$ be the standard basis of $\CC^{2n}$. Consider $\mathcal{C}$, the tensor algebra $T(\CC^{2n})$ modulo the relations 
\[\mathbf v\otimes\mathbf w + \mathbf w\otimes \mathbf v - Q(\mathbf v,\mathbf w).\] 
The resulting algebra $\mathcal{C}$ is called the \emph{Clifford algebra}.
The underlying vector space of $\mathcal{C}$ is graded by $\mathcal{C} = \bigoplus_{i=0}^{2n} \mathcal{C}^i$, where $\mathcal{C}^i$ is spanned by products of $i$ pairwise orthogonal vectors. In particular, the degree 1 part $\mathcal{C}^1$ is $\CC^{2n}$. The degree 2 part $\mathcal{C}^2$ is a Lie algebra under commutator bracket. It has a basis of the form $\frac{1}{2}(\be_i\be_j-\be_j\be_i)_{i<j}$ and acts on $\mathcal{C}^1$ via commutator bracket. This action identifies $\mathcal{C}^2$ with $\mathfrak{so}_{2n}$. Explicitly, if $E_{ij}$ is the matrix with $1$ at position $(i,j)$ and $0$s elsewhere, then the vector $\frac{1}{2}(\be_i\be_j-\be_j\be_i)$ is sent to $E_{i\overline{j}}-E_{j\overline{i}}$.

Let $\CC^{2n}=W\oplus W^*$ be the decomposition of $\CC^{2n}$ into the span of $\be_1,\ldots, \be_n$ and the span of $\be_{n+1},\ldots, \be_{2n}$. We can identify $W^*$ as the dual space of $W$ using the bilinear form $Q$. The spin representation $S = S^-\oplus S^+ = V_{\varpi_{n-1}}\oplus V_{\varpi_n}$ has underlying vector space $\bigwedge^\bullet W = \bigwedge^{\mathrm{odd}}W\oplus \bigwedge^{\mathrm{even}}W$. This space is a module for $\mathcal{C}$, with action induced by $\mathbf w\cdot \omega = \mathbf w\wedge \omega$ for $\mathbf w\in W$ and $\mathbf w^*\cdot\omega = \iota_{\mathbf w^*}\omega$ for $\mathbf w^*\in W^*$. Here $\iota_{\mathbf w^*}$ is the interior product, determined recursively by $\iota_{\mathbf w^*}(1) = 0$ and $\iota_{\mathbf w^*}(\mathbf w\wedge \omega') = (\mathbf w^*,\mathbf w)\omega' - \mathbf w\wedge \iota_{\mathbf w^*}\omega'$ for $\mathbf w\in W$.

If $I = \{i_1 < \cdots < i_r\}$, then write $\be_I = \be_{i_1}\wedge \cdots \wedge \be_{i_n}$. 
We see that $E_{ij}-E_{\overline{ji}}$ in $\mathfrak{so}_{2n}$ acts via $\frac{1}{2}(\be_i\be_{\overline{j}}-\be_{\overline{j}}\be_i)$. 
Hence $\be_I$ is a weight vector of weight $\sum_{i\in I} \frac{1}{2}\be_i - \sum_{i\not\in I}\frac{1}{2}\be_i$, and $\{\be_I \mid I\subseteq [n]\}$ is a weight basis of $S$. The weight spaces associated to $V_{\varpi_{n-1}}$ are those with $|I|$ odd, and the weight spaces associated to $V_{\varpi_n}$ are those with $|I|$ even. Since the basis $\{\be_I\mid I\subseteq [n],~|I|\text{ even}\}$ of $V_{\varpi_n}$ (resp. the basis $\{\be_I\mid I\subseteq [n],~|I|\text{ odd}\}$ of $V_{\varpi_{n-1}}$) is a single $W^{D_n}$-orbit containing a highest weight vector, it coincides with the Lusztig canonical basis.

\subsubsection{The Grassmannian $\operatorname{SOGr}(n,2n)$}
The orthogonal Grassmannian $\operatorname{SOGr}(n,2n)_+$ is (one of two components of) the space of $n$-dimensional isotropic subspaces of $\CC^{2n}$ under the symmetric form defined by $E$. It is the maximal parabolic quotient associated to the simple root $\alpha_{n}$, and hence embeds in $\PP V_{\varpi_{n}}$.  A representative element is the isotropic subspace spanned by $\be_1,\ldots,\be_{n-1},\be_{n}$. We can represent an element of $\operatorname{SOGr}(n,2n)_+$ by a full-rank $2n\times n$ matrix $\begin{bmatrix}A\\B\end{bmatrix}$ satisfying $A^tE_0B+B^tE_0^tA=0$, up to right multiplication by an invertible matrix. We remark that the isotropic subspace spanned by $\be_1,\ldots,\be_{n-1},\be_{\overline{n}}$ is \emph{not} an element of $\operatorname{SOGr}(n,2n)_+$; it is instead in $\operatorname{SOGr}(n,2n)_-$, which we identify as the maximal parabolic quotient associated to the simple root $\alpha_{n-1}$.

\subsubsection{Coordinates on $\operatorname{SOGr}(n,2n)_+$}
\newcommand{\Pf}{\mathrm{Pf}}
\newcommand{\Spin}{\mathrm{Spin}}
The weight vector coordinates on $V_{\varpi_n}=S^+$ induce homogeneous coordinates on $\operatorname{SOGr}(n,2n)_+$. Let $\Pf_I$ denote the coordinate $g\mapsto \langle \be_{[n]\setminus I}, g \be_{[n]}\rangle$, for $I\subseteq [n]$ with $|I|$ even. (Here $\langle \be_{[n]\setminus I},g\be_{[n]}\rangle$ means the $\be_{[n]\setminus I}$-coefficient of $g\be_{[n]}$ in the basis $\{\be_{I}\}$.) If $|I|$ is odd, then set $\Pf_I=0$. If $X=\begin{bmatrix}A\\B\end{bmatrix}$ represents a point in $\operatorname{SOGr}(n,2n)_+$, then the vector $(\Pf_I(X)^2)_{I\subseteq [n]}$ coincides with the vector $(\Delta_{\widetilde I})_{I\subseteq [n]}$ up to rescaling. Here $\Delta_{\widetilde I}$ is the minor of $X$ with rows $\left([n]\setminus I \right)\sqcup \overline{I}$, and $\Pf_I(X)$ is the value of $\Pf_I$ on an element of
(the universal cover $\Spin_{2n}$ lifting an element of) $\operatorname{SO}_{2n}$ of the form $\begin{bmatrix} A & ? \\ B & ? \end{bmatrix}$. These vectors coincide because the highest weight of $V_{\varpi_n}\otimes V_{\varpi_n}$ is the same as the highest weight of $\bigwedge^n V_{\varpi_1}$, so the coordinates with respect to the $W^{D_n}$ orbit of their highest weight vectors are the same. We normalize $\Pf_I(X)$ so that when $X= \begin{bmatrix} I_n \\ B \end{bmatrix}$, then $\mathrm{sgn}(I,[n]\setminus I)\Pf_{I}(X)$ is the Pfaffian of the (antisymmetric) submatrix of $E_0B$ given by selecting the rows and columns indicated by $I$. (Points of this form make up the dense Schubert cell in $\operatorname{SOGr}(n,2n)_+$.)

\subsubsection{Counterexample to $\operatorname{SOGr}(n,2n)_+^{\geq 0} = \operatorname{SOGr}(n,2n)_+^{\mathrm{Pf}\geq 0}$}
\label{sec:typeDcounterexample}
The spin representation $S^+$ is miniscule, so its canonical basis coincides with its weight basis. 
Hence \cite[Chapter 4, Proposition 5.1]{Rietsch} would imply that $\operatorname{SOGr}(n,2n)_+^{\geq 0}$ is the locus of non-negativity of $\{\mathrm{Pf}_I \mid I\subseteq [n],~|I|\text{ even}\}$ in $\operatorname{SOGr}(n,2n)_+$. However, this is not the case: there is an element of $\operatorname{SOGr}(n,2n)_+$ with strictly positive values of its Pfaffians which is not in $\operatorname{SOGr}(n,2n)_+^{\geq 0}$. Indeed, consider the reduced word
\[ w_0^{\{1,\ldots,n-1\}} = s_4s_2s_3s_1s_2s_4,  \]
and the element $y_4(t_1)y_2(t_2)y_3(t_3)y_1(t_4)y_2(t_5)y_4(t_6)P$ in $\operatorname{SOGr}(n,2n)_+$. This element of $\operatorname{SOGr}(n,2n)_+$ is represented by the matrix
\[
\begin{bmatrix}
    1 & 0 & 0 & 0 \\
    0 & 1 & 0 & 0 \\
    0 & 0 & 1 & 0 \\
    0 & 0 & 0 & 1 \\
    t_4t_5t_6 & -(t_2+t_5)t_6 & t_1+t_6 & 0 \\
     t_3t_4t_5t_6 & -t_3t_5t_6 & 0 & t_1+t_6 \\
     t_2t_3t_4t_5t_6 & 0 & -t_3t_5t_6 & (t_2+t_5)t_6 \\
     0 & t_2t_3t_4t_5t_6 & -t_3t_4t_5t_6 & t_4t_5t_6
\end{bmatrix}.
\]
The associated antisymmetric matrix is
\[ E_0B = \begin{bmatrix}
    0 & -t_2t_3t_4t_5t_6 & t_3t_4t_5t_6 & -t_4t_5t_6 \\
    t_2t_3t_4t_5t_6 & 0 & -t_3t_5t_6 & (t_2+t_5)t_6 \\
    -t_3t_4t_5t_6 & t_3t_5t_6 & 0 & -t_1-t_6 \\
    t_4t_5t_6 & -(t_2+t_5)t_6 & t_1+t_6 & 0
\end{bmatrix}. \]
The Pfaffians are $\Pf_{\varnothing}=1$, $\Pf_{\{1,2,3,4\}}=t_1t_2t_3t_4t_5t_6$, and the Pfaffians with index of size $2$, which are (up to sign) the entries of the matrix,
\[t_2t_3t_4t_5t_6,~t_3t_4t_5t_6,~t_4t_5t_6,~t_3t_5t_6,(t_2+t_5)t_6,t_1+t_6. \]
We see that if $t_1,t_2,t_3,t_4$ are positive, $t_5,t_6$ are negative, and $|t_6|\ll t_1$ and $|t_5|\ll t_2$, then all of the canonical basis coordinates are positive. However, this point of $\operatorname{SOGr}(n,2n)_+$ is Lusztig non-negative if and only if $t_i\geq 0$ for all $i$.

\appendix

\section{Other Proofs of (3)\texorpdfstring{$\implies$}{=>}(1)}\label{sect:MRparameterization}

The main body of this paper contains a rather minimalistic and  general proof of \ref{mainthm:consec}$\implies$\ref{mainthm:pos} in \Cref{thm:main}.
In this appendix, we present two other proofs which rely less on general theory and, to varying degrees, more on the specific combinatorial properties of the flag varieties investigated.
Accordingly, these proofs provide more insight into how precisely the types $B$ and $C$ flag varieties embed into the appropriate type $A$ flag variety, which may be of independent interest.
We review relevant background in the next subsection, and present the two proofs in the second and third subsections of this appendix. 

\subsection{Distinguished subexpressions and Deodhar decompositions}

Introduced in \cite{deodhar1,deodhar2} (for complete and partial flag varieties, respectively), the Deodhar decomposition is a refined decomposition of a flag variety.
Marsh and Rietsch gave a combinatorial parametrization for Deodhar components that is particularly well-suited for describing Lusztig positivity \cite{marshrietsch2004}.
We will take their parametric description as the definition of Deodhar components.

\medskip
Let $G$ be a group with a fixed pinning, and as before, let $W$ be its Weyl group with simple reflections $\{s_i: i\in I\}$.
We start with combinatorial preparations.

Let $\bv$ be an expression for $v\in W$ and let $\bu$ be a subexpression of $\bv$ for $u\in W$. For $k\geq 0$, we will let $\bu_{(k)}$ be the subexpression of $v$ which is identical to $\bu$ in its first $k$ entries and is all $1$s afterwards. Let $u_{(k)}\in W$ be the element of $W$ given by the expression $\bu_{(k)}$.

\begin{eg}
    Suppose $W=\mathfrak{S}_4$ and $\bv=s_1s_2s_3s_1s_2$ is an expression for $v=4312$ (in one-line notation). An example of a subexpression is $\bu=s_11s_3s_11$, which is an expression for $u=1243$. We then have $\bu_{(0)}=11111$, $\bu_{(1)}=\bu_{(2)}=s_11111$, $\bu_{(3)}=s_11s_311$, and $\bu_{(4)}=\bu_{(5)}=\bu$. Accordingly, $u_{(0)}=1234$, $u_{(1)}=u_{(2)}=2134$, $u_{(3)}=2143$, and $u_{(4)}=u_{(5)}=u$.
\end{eg}

\begin{defn} \label{defn:Jsets}
For a subexpression $\bu$ of a reduced expression $\bv=s_{i_1}\cdots s_{i_p}$, define 
    \begin{align*}
        J_{\bu}^{+}&:=\{k\in[p]\mid u_{(k)}> u_{(k-1)}\},\\
        J_{\bu}^{\circ}&:=\{k\in[p]\mid u_{(k)}=u_{(k-1)}\}, \text{ and}\\
        J_{\bu}^{-}&:=\{k\in[p]\mid u_{(k)} < u_{(k-1)}\}.
    \end{align*}
We say that the subexpression $\bu$ of $\bv$ is

\emph{distinguished}, denoted $\bu \preceq \bv$, if $u_{(j)}\leq u_{(j-1)}s_{i_j}$ for all $j\in[p]$. That is, if multiplying $u_{(j-1)}$ on the right by $s_{i_j}$ decreases the length of $u_{(j-1)}$, then $\bu$ must contain $s_{i_j}$. We say it is \textit{reverse distinguished} if $u_{(j-1)}\leq s_{i_j}u_{(j)}$. 
\end{defn}

The Deodhar components of the flag variety $G/B$ correspond to the distinguished subexpressions, as follows.

\begin{defn}\label{defn:Deodhar}
Let $\bv=s_{i_1}\cdots s_{i_\ell}$ be a reduced expression for $v\in W$, and let $\bu$ be a distinguished subexpression.
Define a map $p_{\bu,\bv} : \mathbb{R}^{J^-_\bu}\times (\mathbb{R^*})^{J^\circ_\bu}\rightarrow G$ by
\[
p_{\bu,\bv}(\mathbf m, \mathbf t) := g_1 \dotsm g_\ell \quad \text{where}\quad
g_k := \begin{cases}
x_{i_k}(m_k)\dot{s_{i_{k}}}^{-1} &  \text{if } k \in J^-_{\bu}
        \\
y_{i_k}(t_k) & \text{if } k \in J^\circ_{\bu} \text{, and}\\
\dot{s_{i_k}}&  \text{if }  k\in J^+_{\bu}.
\end{cases}
\]
Denote by $G_{\bu,\bv}$ the image of this map, and define the \emph{Deodhar cell} $\mathcal R_{\bu,\bv}$ of $\bu\preceq \bv$ to be the image of $G_{\bu,\bv}$ in $G/B$ under the projection map $G \to G/B$.
\end{defn}

We record the properties of these cells we will need.
The last statement (c) below, which appeared as Lemma~\ref{lem:positiveparam}, follows from (a) by comparing Definitions~\ref{defn:LusztigPos} and \ref{defn:Deodhar}.

\begin{prop}\label{prop:Deodhar} We have the following.
\begin{enumerate}[label = (\alph*)]
\item \cite[Proposition 5.2]{marshrietsch2004}
In the setting of Definition~\ref{defn:Deodhar}, the composition 
$\mathbb{R}^{J^-_\bu}\times (\mathbb{R^*})^{J^\circ_\bu} \to G_{\bu,\bv} \to \mathcal R_{\bu,\bv}$ is a bijection.  We abuse notation and write $p_{\bu,\bv}$ also for this composition.
\item \cite[Section 4.4]{marshrietsch2004}
For each $v\in W$, fix a reduced expression $\bv$.  Then, we have
\[
G/B = \bigsqcup_{v\in W} \bigsqcup_{\bu\preceq \bv} \mathcal R_{\bu,\bv},
\]
called the \emph{Deodhar decomposition} of the flag variety $G/B$.
\item Fix a reduced expression $\mathbf{w_0}$ for the longest element $w_0\in W$, and let $\mathbf 1$ denote the subexpression $(1, \dotsc, 1)$, which is distinguished.  Then, $(G/B)^{>0} = p_{\mathbf 1, \mathbf{w_0}}(\RR_{>0}^{\ell(w_0)}) \subset \mathcal R_{\mathbf 1, \mathbf{w_0}}$.
\end{enumerate}
\end{prop}

\subsection{Proof 1: Reduction to the complete flag case} \label{sec:3implies1}

Like the proof of \ref{mainthm:pos}$\implies$\ref{mainthm:nonneg}, we first work in the general setting as stated in Setup~\ref{setup} involving an embedding $\iota: G \hookrightarrow GL_N$ and a map $\psi$.
Here, for an expression ${\mathbf s}_{\mathbf i}^\Phi$, we write $\psi({\mathbf s}_{\mathbf i}^\Phi)$ for the expression ${\mathbf s}_{\psi(\mathbf i)}^A$.
For a subexpression $\bu = u_1 \dotsm u_p$ of an expression $\bv = s_{i_1}\dotsm s_{i_p}$, we modify $\psi$ to define $\widetilde\psi(\bu) = \widetilde\psi(u_1)\dotsb \widetilde\psi(u_p)$ where
\[
\widetilde\psi(u_k) =
\begin{cases}
\psi(s_{i_k}) & \text{if $k \notin J_\bu^\circ$}\\
(1, \dotsc, 1) \text{ of length $|\psi(i_k)|$} & \text{if $k\in J_\bu^\circ$},
\end{cases}
\]
so that $\widetilde\psi(\bu)$ is a subexpression of $\psi(\bv)$.

\begin{defn}\label{condition2}
We say that the pair $(\iota, \psi)$ has \emph{property ($\ddagger$)} if the following are satisfied:
\begin{itemize}
\item[($\ddagger1$)] The map $\psi$ preserves the property of being reduced and distinguished.  That is, for a reduced expression $\bv$ for $v\in W^\Phi$ and a distinguished subexpression $\bu$ of $\bv$, one has that $\psi(\bv)$ is a reduced expression for an element in $W^A$, and $\widetilde\psi(\bu)$ is a distinguished subexpression of $\psi(\bv)$.
\item[($\ddagger2$)] For every $i\in I$, under the embedding $\iota: G \hookrightarrow GL_N$ we have
\[
x_i^\Phi(m)(\dot s_i^\Phi)^{-1} = x_{i_1}^A(f_1(m)) (\dot s_{i_1}^A)^{-1} \dotsm x_{i_\ell}^A(f_\ell(m)) (\dot s_{i_\ell}^A)^{-1}
\]
where $\psi(i) = (i_1, \dotsc, i_\ell)$ and $(f_1, \dots, f_\ell)$ is a sequence of functions $f_j: \RR \to \RR$.
\item[($\ddagger3$)]
All the functions $f_j: \RR \to \RR$ appearing in \hyperref[condition1]{($\dagger1$)} further satisfy $f_j(\RR_{\leq 0}) \subseteq \RR_{\leq 0}$.
\end{itemize}
\end{defn}

\begin{prop}\label{prop:ddagger}
Suppose the pair $(\iota, \psi)$ satisfies the property \hyperref[condition1]{($\dagger$)}, so that $G/B_+$ naturally embeds in $\operatorname{Fl}_{[N-1];N}$.
If the pair $(\iota, \psi)$ further satisfies the property \hyperref[condition2]{($\ddagger$)}, then
\[
(G/B_+)^{>0} = (G/B_+) \cap \operatorname{Fl}_{[N-1];N}^{\Delta >0}.
\]
\end{prop}

\begin{proof}
By Theorem~\ref{thm:BK} (also \cite[Theorem 4.11]{boretsky2023totally}), we may prove the statement with $\operatorname{Fl}_{[N-1];N}^{\Delta >0}$ replaced by $\operatorname{Fl}_{[N-1];N}^{>0}$.
Fix reduced expressions for elements of $W^\Phi$, with $\mathbf{w_0}$ for the longest element.  By \hyperref[condition1]{($\dagger2$)}, we may fix $\mathbf{w_0}$ such that $\psi(\mathbf{w_0})$ is a reduced expression for the longest element of $W^A$.
For every pair $\bu\preceq \bv$ in $W^\Phi$, the properties \hyperref[condition1]{($\dagger1$)} and \hyperref[condition1]{($\ddagger$)} together  imply that the embedding $\iota$ induces an inclusion $\mathcal R^\Phi_{\bu,\bv} \subseteq \mathcal R^A_{\widetilde\psi(\bu),\psi(\bv)}$.
Since $\mathcal R_{\bu,\bv}^\Phi \subseteq \mathcal R_{\widetilde\psi(\bu),\psi(\bv)}^A$ is disjoint from $\mathcal R_{\mathbf 1, \psi(\mathbf{w_0})}^A$ unless $\bu = \mathbf 1$ and $\psi(\bv) = \psi(\mathbf{w_0})$ by Proposition~\ref{prop:Deodhar}(b),
and since $\psi(\bv) = \psi(\mathbf{w_0})$ implies that $\bv$ is a reduced expression for the longest element of $W^\Phi$ by \hyperref[condition1]{($\dagger2$)} and \hyperref[condition2]{($\ddagger1$)}, 
we find that $\mathcal R_{\mathbf 1, \mathbf{w_0}}^\Phi = (G/B) \cap \mathcal R_{\mathbf 1, \psi(\mathbf{w_0})}^A$.
The desired statement now follows from the property \hyperref[condition2]{($\ddagger$3)}, Proposition~\ref{prop:Deodhar}(a), and Proposition~\ref{prop:Deodhar}(c).
\end{proof}

We now proceed to prove Theorem~\ref{thm:main}\ref{mainthm:consec}$\implies$\ref{mainthm:pos}.
Let us prepare by recording two properties concerning subexpressions when $\Phi$ is either $C$ or $B$, which amount to stating that the property \hyperref[condition2]{($\ddagger1$)} is satisfied. We abuse notation by using $\iota$ and $\psi$ for both the maps starting from type $C$ objects and those starting from type $B$ objects.

\begin{prop}\label{prop:foldingnondecreasing}
    Let $\bu$ be a reduced subexpression in $\bv$, for $\bv$ a reduced expression for some $v\in W^\Phi$. Then $\tilde{\psi}(\bu)$ is a reduced subexpression in $\psi(\bv)$.  
\end{prop}

\begin{proof}
    We may ignore any factors that are $1$s. The result then follows from \Cref{prop:psipreservesreduced}.
\end{proof}

\begin{prop}\label{prop:foldingdistinguished}
    Let $\bu$ be a distinguished subexpression in $\bv$, for $\bv$ a reduced expression for some $v\in W^\Phi$. Then $\tilde{\psi}(\bu)$ is a distinguished subexpression in $\psi(\bv)$. 
\end{prop}

\begin{proof}
This follows from either \Cref{cor:reducedinAfromC} or \Cref{cor:reducedinAfromB}, depending on whether $\Phi=C$ or $\Phi=B$. 
\end{proof}

\begin{proof}[Proof of Theorem~\ref{thm:main} \ref{mainthm:consec}$\implies$\ref{mainthm:pos}]
Let $n\geq 2$.
As the embeddings $\operatorname{Sp}_{2n} \hookrightarrow GL_{2n}$ and $\operatorname{SO}_{2n+1} \hookrightarrow GL_{2n+1}$ satisfy \hyperref[condition1]{($\dagger$)}, Definition~\ref{defn:LusztigPos} and Proposition~\ref{prop:dagger}(a) together imply that $\operatorname{SpFl}_{K ; 2n}^{>0} \subseteq \operatorname{SpFl}_{K; 2n}^{\Delta>0}$ and $\operatorname{SOFl}_{K ; 2n+1}^{>0} \subseteq \operatorname{SOFl}_{K; 2n+1}^{\Delta>0}$ for any $K\subseteq [n]$.
It remains to show the reverse inclusions when $K = \{k, k +1, \dotsc, n\}$ for some $k\in [n]$.

We first reduce to the case $K = [n]$ as follows.
We show this reduction for the $\operatorname{Sp}_{2n}$ case; the case of $\operatorname{SO}_{2n+1}$ is similar.
By Corollary~\ref{cor:pinning3}, a point in $\operatorname{SpFl}_{K; 2n}^{\Delta>0}$ is a point $L_\bullet$ in $(\operatorname{Sp}_{2n}/P_K^C)\cap \operatorname{Fl}_{K\cup (2n-K); 2n}^{\Delta >0}$.
Since $K\cup (2n-K)$ consists of consecutive integers by our assumption on $K$, Theorem~\ref{thm:BK} implies $ \operatorname{Fl}_{K\cup (2n-K); 2n}^{\Delta >0} =  \operatorname{Fl}_{K\cup (2n-K); 2n}^{>0}$.  Since by definition $\operatorname{Fl}_{K\cup (2n-K); 2n}^{>0}$ is the projection of $\operatorname{Fl}_{[2n-1]; 2n}^{>0}$, we may extend the flag $L_\bullet$ to a flag $\widetilde L_\bullet$ in $\operatorname{Fl}_{[2n-1]; 2n}^{>0} = \operatorname{Fl}_{[2n-1]; 2n}^{\Delta>0}$.
Because subspaces of isotropic subspaces are isotropic, the projection of $\widetilde L_\bullet$ to $\operatorname{Fl}_{[n];2n}$ is a point in $\operatorname{SpFl}_{[n];2n}^{\Delta >0}$.
In particular, by Lemma~\ref{lem:pinning2}, we may choose $\widetilde L_\bullet$ such that $\widetilde L_\bullet \in (\operatorname{Sp}_{2n}/B_+) \cap \operatorname{Fl}_{[2n-1]; 2n}^{\Delta>0}$.
Hence, if Lusztig positivity and Pl\"ucker positivity agrees for the case of $K = [n]$, then $\widetilde L_\bullet \in (\operatorname{Sp}_{2n}/B)^{>0}$ so that its projection $L_\bullet$ is Lusztig positive also.

Lastly, the case of $K = [n]$ follows from Proposition~\ref{prop:ddagger} and Corollary~\ref{cor:pinning3} once we show that the embeddings $\operatorname{Sp}_{2n} \hookrightarrow GL_{2n}$ and $\operatorname{SO}_{2n+1} \hookrightarrow GL_{2n+1}$ satisfy the property \hyperref[condition2]{($\ddagger$)} in addition to the property \hyperref[condition1]{$(\dagger$)}.
The property \hyperref[condition1]{$(\dagger$)} was already verified in the proof of \ref{mainthm:pos}$\implies$\ref{mainthm:nonneg}.  For the property \hyperref[condition2]{($\ddagger$)}, \Cref{prop:foldingnondecreasing} and \Cref{prop:foldingdistinguished} implies \hyperref[condition2]{($\ddagger1$)}, and the following observations, whose verification is straightforward from the explicit descriptions of the pinnings, imply \hyperref[condition2]{($\ddagger2$)}:
\begin{itemize}
\item For any $i\in [n-1]$, we have that any element of $\{x_i^A(m), \dot s_i^A\}$ commutes with any element of $\{x_{2n-i}^A(m), \dot s_{2n-i}^A\}$ in $GL_{2n}$, and similarly for the pairs $\{x_i^A, \dot s_i^A\}$ and $\{x_{2n+1-i}^A, s_{2n+1-i}^A\}$ in $GL_{2n+1}$. 
\item We have
$
x_n^B(m)(\dot s_n^B)^{-1} = x_n^A(\sqrt{2}m) (\dot s_n^A)^{-1}x_{n+1}^A(-m^2) (\dot s_{n+1}^A)^{-1}x_n^A(\sqrt{2}m) (\dot s_n^A)^{-1}.
$
under the embedding $\operatorname{SO}_{2n+1} \hookrightarrow GL_{2n+1}$.
\qedhere
\end{itemize}
\end{proof}

\begin{rem}
Condition \hyperref[condition2]{($\ddagger2$)} can be weakened to a form that is more general than what we need for our purposes. Let $\iota$ be an embedding of $G$, with root system $\Phi$, into $\mathfrak{S}_N$. Let $\psi$ be the corresponding map of sequences of $[N-1]$. For an interval $I=\{a,a+1,\ldots, a+b\}\subset [N]$, denote by $w_0^I$ the permutation in $\mathfrak{S}_N$ which fixes $[N]\setminus I$ and maps $a+i$ to $a+b-i$ for $0\leq i\leq b$. Suppose that, in addition to \hyperref[condition1]{($\dagger$)} and \hyperref[condition2]{($\ddagger1$)}, $(\iota,\psi)$ satisfies ($\ddagger2'$): For each $s_i^\Phi$, $\psi(s_i^\Phi)=\prod w_0^I$, where the product is over disjoint intervals of $[N]$.  Then, $\mathcal{R}_{\bu,\bv}^\Phi\subseteq \mathcal{R}_{\tilde{\psi}(\bu),\psi(\bv)}^A$. In particular, this implies that \Cref{prop:ddagger} still holds. We include a proof sketch.
\end{rem}
\begin{proof}[Proof sketch.]
We observe that $y_i^\Phi(t)=y_{\psi(i)}^A(\mathbf{t})$, where $\mathbf{t}=(f_1(t),f_2(t),\ldots, f_{|\psi(i)|}(t))\in \mathbb{R}^{\ell(\psi(i))}$, where the $f_i$ are as in \hyperref[condition1]{($\dagger1$)}. Also, $\dot{s}_i^\Phi=\dot{s}_\mathbf{i}^A$, and $x_i^\Phi(t) (\dot{s}_i^\Phi)^{-1}=D(x_i^A(t_j) (\dot{s}_i^A)^{-1})_\mathbf{i}$ for some diagonal matrix $D$ with $\pm1$ on the diagonal and some choice of $\{t_j\}$. The first two observations are immediate. We provide more detail on the third. 

Since permutation of disjoint subsets commute, it suffices to focus on the case that $\psi(s_i^\Phi)=\mathbf{s}_{\mathbf{i}}^A$ is an expression for $w_0^A$ in $\mathfrak{S}_n$ for some $n\leq N$. Let $G\coloneqq x_i^\Phi(t) (\dot{s}_i^\Phi)^{-1}=x^A_\mathbf{i}(\mathbf{t})(\dot{s}_\mathbf{i}^A)^{-1}$, with $\mathbf{t}$ as in the previous paragraph. The first factor is upper triangular. The second is anti-diagonal with $\pm1$ on the anti-diagonal. Thus, $G$ is a matrix with $\pm1$ on the anti-diagonal and any non-zero entry on or above the anti-diagonal. On the other hand, using the explicit form $x_i^A(t)(\dot{s}_i^A)^{-1}=\phi_i\left(\begin{bmatrix}
    -t & 1 \\
    -1 & 0
\end{bmatrix}\right)$, we compute $H\coloneqq (x_i^A(t_{j})(\dot{s}_i^A)^{-1})_\mathbf{i}$. To do this, we observe the following braid move: \[x_i^A(t_{1})(\dot{s}_i^A)^{-1}x_{i+1}^A(t_{2})(\dot{s}_{i+1}^A)^{-1}x_i^A(t_{3})(\dot{s}_i^A)^{-1}=x_{i+1}^A(t_{3})(\dot{s}_{i+1}^A)^{-1}x_i^A(t_{1}t_3-t_2)(\dot{s}_i^A)^{-1}x_{i+1}^A(t_{1})(\dot{s}_{i+1}^A)^{-1}.\] This allows us to assume without loss of generality that $\mathbf{i}=(1,2,\ldots, n-1,1,2,\dots, n-2,\ldots, 1,2,1)$. Then, by an explicit computation by induction, we see that $H$ has entries $\pm 1$ on the anti-diagonal and, by making appropriate choices of $\{t_j\}$, can have any entries we desire above the anti-diagonal. Thus, for some choice of $\{t_j\}$, $G=HD$, where $D$ is a diagonal $\pm 1$ matrix chosen such that the anti-diagonals of $G$ and $DH$ agree.

We recall the notation $g_k$ of \Cref{defn:Deodhar}. We will use superscripts to denote a root system for $g_k$, as we do for $x_i$, $y_i$ and $\dot{s}_i$. Consider a flag $F\in\mathcal{R}_{\bu,\bv}^\Phi$. Then $F$ is represented by some $p_{\bu,\bv}(\mathbf{m},\mathbf{t})=g_1^\Phi\cdots g_{\ell(\bv)}^\Phi$. We have shown that this matrix equals $g_1^Ad_1g_2^Ad_2\cdots g_{\ell(\psi(\bv))d_{\ell(\psi(\bv))}}^A$, where each $d_i^A$ is diagonal with $\pm 1$ on the diagonal (many of these are identity matrices, but we may get non-identity matrices for each $g_k^\Phi=x_i^\Phi(\dot{s}_i)^{-1}$). It is straightforward to see that this can be rewritten as $g_1^Ag_2^A\cdots g_{\ell(\psi(\bv))}^AD'$, for some $D'$ diagonal with $\pm 1$ on the diagonal. However, multiplying by a diagonal matrix on the right does not change $g_1^Ag_2^A\cdots g_{\ell(\psi(\bv))}^A$ as a flag. By definition, the latter lies in $\mathcal{R}_{\tilde{\psi}(\bu),\psi(\bv)}^A$. Thus, $\mathcal{R}_{\bu,\bv}^\Phi\subseteq\mathcal{R}_{\tilde{\psi}(\bu),\psi(\bv)}^A$.
\end{proof}

\subsection{Proof 2: Relating distinguished subexpressions in \texorpdfstring{$W^\Phi$ and $W^A$}{W\^{Phi} and W\^A}}

This subsection contains a proof of \ref{mainthm:consec}$\implies$\ref{mainthm:pos}, which is based primarily on the combinatorics of expressions in $W^\Phi$. For the rest of this section, fix some $K=\{k,\cdots, n\}$ and the corresponding $J=\{1,2,\cdots, k-1\}$. If $\Phi=C$, define $\hat{J}\coloneqq \{i\mid 2n-i\in J\}$. If $\Phi=B$, define $\hat{J}\coloneqq \{i\mid 2n+1-i\in J\}$. In either case, let $J^A=J\cup \hat{J}$ and define $K^A$ similarly. Finally, we will use $J'=[2n-1]\setminus K$ if $\Phi=C$ and $J'=[2n]\setminus K$ if $\Phi=B$.

\textbf{Notation:} In this section, we will have expressions for many different Weyl group elements appearing. In order to help us keep track of them, we will use the following notation: for $W$ a Weyl group, $\mathbf{w}=\mathbf{s}_\mathbf{i}$ an expression for $w\in W$, and $\mathbf{a}\in \mathbb{R}^{\ell(\mathbf{w})}$, we will write $y_\mathbf{w}(\mathbf{a})\coloneqq y_\mathbf{i}(\mathbf{a})$.

The combinatorics of partial flag varieties is related to the combinatorics of parabolic subgroups of a Weyl group, defined as follows: 
\begin{defn}
Let $W$ be a Weyl group and $J$ a subset of the roots of the corresponding root system. We let $W_J=\langle s_i\mid i\in J\rangle$ be the corresponding parabolic subgroup. We denote by $W^J$ the quotient $W/W_J$.    
\end{defn}

We recall the definition of a descent in a Weyl group $W$.

\begin{defn}
    We say $w$ has a \textit{descent} at position $i$ if the pair $(i,i+1)$ is an inversion. We denote by $\textnormal{des}(w)$ the set of positions of its descents. 
\end{defn}

\begin{prop}
For a subset $J\subseteq [n]$, let $w^J$ be the minimal length coset representative for $w\in W$ in $W^J$. Then, $\textnormal{des}(w^J)\subset [n]\setminus J$.
\end{prop}

We present the partial flag version of \Cref{prop:Deodhar}.

\begin{defn} Define the \emph{Deodhar cell} $\mathcal R_{\bu,\bv}^J$ of $\bu\preceq \bv$ to be the image of $G_{\bu,\bv}$ in $G/{P_J}$ under the projection map $G \to {G/P_J}$.
\end{defn}

\begin{prop}\label{prop:Deodharpartial} We have the following.
\begin{enumerate}[label = (\alph*)]
\item \cite[Section 3.4]{kodamawilliams}
Let $\bv$ be an expression for a minimal length coset representative in $W^J$. In the setting of Definition~\ref{defn:Deodhar}, the composition 
$\mathbb{R}^{J^-_\bu}\times (\mathbb{R^*})^{J^\circ_\bu} \to G_{\bu,\bv} \to \mathcal R^J_{\bu,\bv}$ is a bijection.  We abuse notation and write $p_{\bu,\bv}$ also for this composition.
\item \cite[Section 3.4]{kodamawilliams}
For each $v$ a minimal length coset representative in $W^J$, fix a reduced expression $\bv$.  Then, we have
\[
G/B = \bigsqcup_{v\in W^J} \bigsqcup_{\bu\preceq \bv} \mathcal R_{\bu,\bv},
\]
called the \emph{Deodhar decomposition} of the flag variety $G/B$.
\item Fix a reduced expression $\mathbf{w_0^J}$ for the minimal length coset representative of the longest element $w_0^J\in W^J$, and let $\mathbf 1$ denote the subexpression $(1, \dotsc, 1)$, which is distinguished.  Then, $(G/P_J)^{>0} = p_{\mathbf 1, \mathbf{w_0}^J}(\RR_{>0}^{\ell(w_0^J)}) \subset \mathcal R^J_{\mathbf 1, \mathbf{w_0^J}}$.
\end{enumerate}
\end{prop}

We now transition to discussing the combinatorics of expressions and distinguished subexpressions in $W^\Phi$ and $W^A$, which will lie at the heart of our proof. 

\begin{prop}\cite[Lemma 3.5]{marshrietsch2004}
If $v<w\in W$ and $\bw$ is an expression for $w$, then there is a unique reduced distinguished subexpression $\bv$ for $v$ in $\bw$ and a unique reduced reverse distinguished subexpression $\bv$ for $v$ in $\bw$.    
\end{prop}

Accordingly, we will refer to \textit{the} reduced distinguished subexpression and \textit{the} reduced reverse distinguished subexpression going forward. 

\begin{rem}
    The reduced distinguished and reduced reversed distinguished subexpression of $\bw$ for $v$ can be informally thought of as the rightmost and left most reduced subexpression of $\bw$ for $v$, respectively.
\end{rem}

We leave it to the reader to verify that the following is an expression for $w_0^\Phi\in W^C\cong W^B \cong \mathfrak{S}^\pm_n$: 

\begin{equation*}
\mathbf{w_0}^{\Phi}=s_n^\Phi  (s_{n-1}^{\Phi})   s_n^\Phi(s_{n-2}^\Phi s_{n-1}^{\Phi})s_n(\cdots)s_n^\Phi(s_{1}^{\Phi}s_{2}^{\Phi}\cdots s_{n-1}^\Phi)s_n^\Phi (s_1^\Phi)(s_2^\Phi s_1^\Phi)(\cdots)(s_{n-1}^\Phi\cdots s_1^\Phi),
\end{equation*}

\begin{eg}
    Let $n=4$. Then $\bw_0^\Phi=s_4^\Phi (s_3^\Phi) s_4^\Phi( s_2^\Phi s_3^\Phi)s_4^\Phi (s_1^\Phi s_2^\Phi s_3^\Phi)s_4^\Phi (s_1^\Phi)(s_2^\Phi s_1^\Phi)( s_3^\Phi s_2^\Phi s_1^\Phi)$.
\end{eg}

Denote by $(w_0^J)^{\Phi}$ the minimal coset representative of $w_0^\Phi$ in $(W^C)^J \cong (W^B)^J\cong \mathcal{B}^J$. Denote the reduced reverse distinguished subexpression for $(w_0^J)^{\Phi}$ in $\mathbf{w_0}^{\Phi}$ by $(\mathbf{w_0^J})^\Phi$. By a slight abuse of notation, for a Weyl group element $w$ with expression $\bw$, let $\psi(w)$ denote the Weyl group element given by the expression $\psi(\bw)$.

\begin{rem}\label{rem:descentscomparison}
    Observe that in general, $(w_0^J)^A$, $(w_0^{J^A})^A= \psi((w_0^J)^\Phi)$ and $\left(w_0^{J'}\right)^A$ are all pairwise different. If $w\in W^\Phi$ has $\textnormal{des}(w)=J \subset [n]$, then $\textnormal{des}(\psi(w))=J^A$, and $J\subset J^A\subset J'$. We will write $\psi((w_0^J)^\Phi)$ rather than $(w_0^{J^A})^A$ going forward. The relation between $\psi((w_0^J)^\Phi)$ and $\left(w_0^{J'}\right)^A$ is expanded on in \Cref{cor:Bruhatcomparison}.
\end{rem}

\begin{prop}\label{prop:samefirstentries}
    For $i\in [n]$, we have $(w_0^J)^\Phi(i)=\left(w_0^{J'}\right)^A(i)\in[\overline{n}]$.
\end{prop}

\begin{proof}
    We can write $w_0^{\Phi}=(w_0^J)^{\Phi}(({w_0})_J)^{\Phi}$, where $(({w_0})_J)^{\Phi}$ consists of only $s_i$ with $i\in J$. Since $n\notin J$ and $w_0^{\Phi}([n])=[\overline{n}]$, we must have that $({w_0}^J)^{\Phi}([n])=[\overline{n}]$. Similarly, we can write $w_0^{A}=\left(w_0^{J'}\right)^{A}(({w_0})_{J'})^{A}$, where $(({w_0})_{J'})^{A}$ consists of only $s_i$ with $i\in J'$. Since $n\notin J'$ and $w_0^{A}([n])=[\overline{n}]$, we must have that $({w_0}^{J'})^{A}([n])=[\overline{n}]$. Thus, $({w_0}^J)^{\Phi}([n])=({w_0}^{J'})^{A}([n])=[\overline{n}]$. 
    
    Thought of as a permutation of $[\overline{n}]$, with the ordering induced from our ordering of $[2n]$ or $[2n+1]$ (depending on $\Phi)$, these must both be Bruhat-maximal subject to the condition that descents can only occur at locations $K=[n]\setminus J=[n]\setminus J'$ and so they must coincide.

\end{proof}

\begin{prop}\label{prop:generalbruhatcomparison}
    Let $W$ be a Weyl group and let $R\subset R'$ be sets of roots in the corresponding root system $R_0$. Then, in $W^R$, $w_0^R=w_0^{R'}u$ where $u\in W_{R'}$ 
\end{prop}
\begin{proof}

This follows from \cite[Corollary 2.4.6]{BjornerBrenti}.
\end{proof}
\begin{cor}\label{cor:Bruhatcomparison}
    We have $ \psi(\left(w_0^J\right)^\Phi)=\left(w_0^{J'}\right)^Au$ for some $u\in W^A_{J'}$. 
\end{cor}

\begin{proof}
    This follows from the fact that $ \psi(\left(w_0^J\right)^\Phi)=\left(w_0^{J^A}\right)^A$ and $J^A\subset J'$, together with \Cref{prop:generalbruhatcomparison}.
\end{proof}

In particular, \Cref{cor:Bruhatcomparison} implies that $ \psi(\left(w_0^J\right)^\Phi)\geq \left(w_0^{J'}\right)^A$. Thus, we can define $(\mathbf{w_0^{J'}})^{A}$ to be the reduced reverese distinguished subexpression for for $(w_0^{J'})^A$ in $\psi(\left(\mathbf{w_0^J}\right)^\Phi)$.

\begin{prop}\label{prop:prefix}
    We can obtain an expression $(\mathbf{w_0^J})^\Phi\bu$ from $\mathbf{w_0}^\Phi$ by commutation moves and $(\mathbf{w_0^J})^\Phi$ contains the prefix $\mathbf{p}=s_n^\Phi (s_{n-1}^{\Phi}\cdots s_1^\Phi)s_n^{\Phi}(s_{n-1}^{\Phi}\cdots s_{2}^{\Phi})s_n^\Phi \cdots s_n^\Phi (s_{n-1}^\Phi)s_n^\Phi$ of $\mathbf{w_0}^{\Phi}$. The expression $\bu$ consists of $s_i^\Phi$ with $i\in J$.

\end{prop}

\begin{proof}
    Since $w_0^\Phi=(w_0^J)^\Phi({w_0}_J)^\Phi $, $\mathbf{u}$ is necessarily a reduced expression for $({w_0}_J)^\Phi$. Therefore, it consists of $s_i^\Phi$ for $i\in J$. Since $J=\{1,\ldots, k-1\}$, $W_J^\Phi\cong \mathfrak{S}_{k}$, where the longest element has length $\binom{k}{2}$ \cite{BjornerBrenti}. By direct observation of the expression $\mathbf{w_0}^\Phi$, it is possible to move $\binom{k}{2}$ many $s_i^\Phi$ to the end by commutation moves, namely the $s_{k-1}^\Phi\cdots s_{1}^\Phi$ from the last set of parentheses, the $s_{k-2}^\Phi\cdots s^\Phi_{1}$ from the second to last set of parentheses, and so on, until taking the $s^\Phi_1$ from the $(k-1)-$st to last set of parentheses.  
\end{proof}

\begin{prop}\label{prop:prefixtypeA}
    Let $\Phi=C$. Then, $(\mathbf{w_0^{J'}})^A$ contains the prefix $\psi(\mathbf{p})$, where $p$ is as in \cref{prop:prefix}. 

\end{prop}

\begin{proof}

By \Cref{prop:samefirstentries}, $({w_0}^J)^{A}([n])=[\overline{n}]$. The shortest element of $\mathfrak{S}_{2n}$ with this property is $\psi(\mathbf{p})$ and so the claim follows from the definition of reverse positive distinguished. 
\end{proof}

\begin{prop}\label{prop:distinguishedbraid}
    Let $\Phi=B$, and let $\bu$ be a reduced distinguished subexpression of $\psi(\left(w_0^J\right)^B)$. Let $i\leq n$. It is impossible to rewrite $\bu$ using only commutation moves in $\mathfrak{S}_{2n+1}$ as $\mathbf{v_1} s_{i}^A s_{i+1}^A s_{i}^A\mathbf{v_2}$, with the second $s_i^A$ coming from $\psi(\mathbf{p})$ and $\mathbf{v_1}$, $\mathbf{v_2}$ freely chosen expressions.
\end{prop}

\begin{proof}

    Suppose $i\leq n$ and $s_i^A$ comes from $\psi(\mathbf{p})$. Note that by \cref{prop:prefix}, all of $\psi(\mathbf{p})$ appears in $\psi(\left(w_0^J\right)^B)$. By the description of $\psi$ before \cref{lem:pinning1}, in $\mathbf{w_0}^A$, $s_i^A$ always appears, up to commutation moves, immediately before an $s_{i+1}^A$. Thus, if we apply a braid move $\mathbf{t}=s_i^As_{i+1}^As_i^A\eqqcolon t_1t_2t_3\mapsto \mathbf{r}=s_{i+1}^As_i^As_{i+1}^A\eqqcolon r_1r_2r_3$, we maintain the property of being a subexpression of $\mathbf{w_0}^A$: $r_1$ and $r_2$ take the places of $t_2$ and $t_3$, respectively, whereas $r_3$ in the position of the $s_{i+1}^A$ following $t_3$, guaranteed by the beginning of this proof. We note that this transposition is not already used in $\bu$ since, if it were, we would have a contradiction to the reducedness of $\bu$. Observe that adding $r_3$ to $\bu$ would result in a non-reduced subexpression of $\mathbf{w_0}^A$. This would imply, by the definition of distinguishedness, that $r_3$ should appear in $\bu$, a contradiction.
\end{proof}

\begin{lem}\label{lem:nobadbraid}

    Let $\bw$ be a reduced expression for $w\in\mathfrak{S}_n$. Suppose in $\bw$ there are two factors $r_1,r_2=s_{i+1}^{A}$ with no $s_{i}^A$ between them. Then, up to commutation moves, there is a subsexpression of $\bw$ of the form $\bw_1s_{j}^As_{j+1}^As_{j}^A\bw_2$, with the three simple transpositions lying weakly between $r_1$ and $r_2$ in $\bw$. 
\end{lem}

\begin{proof}

    We work by reverse induction on $i$. If $i=n-1$, the hypotheses are impossible so the lemma statement is vacuously true. 
     
    By reducedness, there is either an $s_i^A$ or an $s_{i+2}^A$ between $r_1$ and $r_2$. By assumption, it is $r_3=s_{i+2}^A$. If this is the only one, we can rewrite $\bw$ as $\bw_1 r_1r_3r_2\bw_2$, which is in the desired form. Otherwise, let $r_3$ be the leftmost of the $s_{i+2}^A$ occurring between $r_1$ and $r_2$, and $r_4$ the second from the left. If there were a factor $r_5=s_{i+1}^A$ between $r_3$ and $r_4$, we could rewrite $\bw$ as $\bw_1 r_1r_3r_5\bw_2$, which is in the desired form. Otherwise, there is no copy of $s_{i+1}$ between $r_3$ and $r_4$, and we are done by induction.
\end{proof}

\begin{defn}
    Define $\psi_{\leq n}(s_i^\Phi)$ to be the subexpression of $\psi(s_i^\Phi)$ consisting of the unique factor $s_i^A$ with $i<n$. Extend $\psi_{\leq n}$ to expressions by $\psi_{\leq n}(s_{i_1}\cdots s_{i_t})=\psi_{\leq n}(s_{i_1})\cdots \psi_{\leq n}(s_{i_l})$.
\end{defn}

\begin{prop}\label{prop:keepsmalltranspositions}
    If $n\notin J$, then $\left(\mathbf{w_0^{J'}}\right)^A$ contains $\psi_{\leq n}((\mathbf{w_0^J})^\Phi)$ as a subexpression.    
\end{prop}

\begin{proof}

Observe that $\psi_{\leq n}((\mathbf{w_0^J})^\Phi)$ consists of all transpositions $s_i^A$ appearing in $\psi((\mathbf{w_0^J})^\Phi)$ with $i\in[n]$. There are exactly $\ell(({w_0^J})^\Phi)$ many of these, one each of the form $\psi_{\leq n}(s_i^\Phi)$ for $s_i^\Phi$ appearing in $(\mathbf{w_0^J})^\Phi$. By \Cref{cor:Bruhatcomparison}, we have a subexpression $\left(\mathbf{w_0^{J'}}\right)^A$ for $\left(w_0^{J'}\right)^A$ in $\psi((\mathbf{w_0^J})^\Phi)$. It suffices to show that this subexpression uses all $\ell((w_0^J)^\Phi)$ of the $s_i^A$ with $i\in[n]$ appearing in $\psi((\mathbf{w_0^J})^\Phi)$.

For $w$ in $W^A$, define $\tinv(w)$ to be the number of inversions plus the number of barred elements amongst the first $n$ entries of $w$. Explicitly, $\tinv(w)\coloneqq|\{(i,j)\in[n]^2\mid i<j,\, w(j)<w(i)\}|+|\{i\in[n]\mid w(i)\in[\overline{n}]\}$. For $w\in W^\Phi$, define $\tinv(w)\coloneqq \tinv(\psi(w))$ Observe that if $w\in W^A$ and $s_i^Aw>w$, then $\tinv(s_i^Aw)\leq \tinv(w)+1$. In particular, note that $\tinv(s_i^Aw)= \tinv(w)$ if $i> n$. Similarly, if $w\in W^\Phi$, then $\tinv(s_i^\Phi w)\leq \tinv(w)+1$. However, observe that $\tinv(w_0^\Phi)=\binom{n+1}{2}=\ell(w_0^\Phi)$. Thus, each $s_i^\Phi$ in an expression for $w_0^\Phi$ contributes exactly $1$ to $\tinv(w_0^\Phi)$. By writing $\mathbf{w_0}^{\Phi}=\bu ((\mathbf{{w_0}^J)})^{\Phi}$, where $\bu$ is an expression for $({w_0}_J)^{\Phi}$, we can conclude that $\tinv((w_0^J)^\Phi)=\ell((w_0^J)^\Phi)$. 

By \Cref{prop:samefirstentries},  $\left(w_0^{J'}\right)^A(i)=\psi((w_0^J)^\Phi)(i)$ for $i\in[n]$ and so $\tinv(\left(w_0^{J'}\right)^A)=\tinv((w_0^J)^\Phi)$. Thus, $\left(w_0^{J'}\right)^A$ has at least $\ell((w_0^J)^\Phi)$ factors of $s_i$ with $i\in[n]$. However, as we observed at the beginning of this proof, that is in fact all of them. 
\end{proof}

\begin{prop}\label{prop:braidmovesbiggerthann}
    An expression $(\mathbf{w_0^{J'}})^A\bu$ can be obtained from $\psi(\mathbf{w_0^J})^\Phi$ without doing any braid moves involving $s_i^A$ for $i\in[n]$.
\end{prop}

\begin{proof}

    We first observe that an expression $(\mathbf{w_0^{J'}})^A\bu$ for $\psi((w_0^J)^\Phi)$ in fact exists by \Cref{cor:Bruhatcomparison}. 
    
    Fix any expression $\bu$ for the Weyl group element $u$ appearing in \Cref{cor:Bruhatcomparison}. We will prove the result by multiplying $\psi(\mathbf{w_0^J})^\Phi$ on the right by $\bu^{-1}$, transposition by transposition. We will move each such transposition $\mathbf{s}$ through $\psi(\mathbf{w_0^J})^\Phi$ as far as possible using only commutation moves until $\mathbf{s}$ either cancels, or cannot move any further without preforming a braid move. A braid move involving $\mathbf{s}$ will be of the form $\mathbf{s}\mathbf{t}\mathbf{s}\mapsto \mathbf{t}\mathbf{s}\mathbf{t}$ for transpositions $\mathbf{s}$ and $\mathbf{t}$ which do not commute. After preforming the braid move, we move the leftmost copy of $\mathbf{t}$ left by commutation moves until it cancels or we are forced to preform another braid move. Eventually, this process must terminate with a cancellation since multiplication by $\mathbf{u}^{-1}$ decreases the length of our permutation by $\ell(\mathbf{u})$. If none of the braid moves performed in this process involve $s_i^A$ for $i\in[n]$, then the result holds. Note that $\psi(\mathbf{w_0^J})^\Phi$ is by definition a reverse distinguished subexpression of itself and this process preserves reverse distinguishedness, so it will leave us with the reduced reverse distinguished subexpression for $({w_0^{J'}})^A$ in $\psi(\mathbf{w_0^J})^\Phi$, namely, $(\mathbf{w_0^{J'}})^A$.

     The key observation for this proof will be that, by \Cref{prop:keepsmalltranspositions}, whatever transposition ends up actually cancelling must be an $s_i^A$ with $i > n$. 
    
    We prove this result inductively on transpositions appearing in $\mathbf{u}^{-1}$. Then, we may assume by induction that we have a reduced expression $\bw$ for some word $w$ such that $(w_0^J)^A<w\leq \psi((w_0^J)^\Phi)$. We move some $s_i^A$ which appears in $\bu$ through it, from right to left, until it cancels out with another transposition. 
    
    By \Cref{cor:Bruhatcomparison}, $\mathbf{u}$ is a product of transpositions $s_{j}^A$ with $j>n$ and so we may assume the $s_i^A$ which we move through $\bw$ has $i>n$. If we cannot commute $s_i^A$ to the left at some point, and it still has not been cancelled out, then we must do a braid move, either of the form $s_i^As_{i-1}^As_{i}^A \mapsto s_{i-1}^As_{i}^As_{i-1}^A$ or of the form $s_i^As_{i+1}^As_{i}^A\mapsto s_{i+1}^As_{i}^As_{i+1}^A$. In the either case, we may now continue to move the new leftmost transposition to the left until it cancels with something or we are forced to do another braid move.

    Suppose we eventually perform a braid move involving an $s_{i}^A$ for $i\leq n$. In particular, along the way, we must perform the braid move $s_{n+1}^As_n^As_{n+1}^A\mapsto s_{n}^As_{n+1}^As_{n}^A$.

     Let $\Phi=C$. All copies of $s_n^A$ occur in the prefix $\mathbf{p}$ of $\mathbf{w_0}^{A}$.
    Since we preformed a braid move involving terms in the prefix $\mathbf{p}$, whatever ends up cancelling must also be in $\mathbf{p}$. This contradicts \cref{prop:prefixtypeA}.
    
    Let $\Phi=B$. We know that whatever eventually cancels out must be a transposition $s_i^A$ with $i>n$. Thus, we must eventually have a braid move of the form $\mathbf{r}=s_{j}^As_{j+1}^As_{j}^A\eqqcolon r_1r_2r_3\mapsto \mathbf s_{j+1}^As_{j}^As_{j+1}^A$ with $j\leq n$. Consider the first such braid move. It must be preceded by the braid move $s_{j+1}^As_{j}^As_{j+1}^A\mapsto \mathbf{q}=\mathbf s_{j}^As_{j+1}^As_{j}^A\eqqcolon q_1q_2q_3$, with no other braid moves in between. Observe that by this construction, $q_1=r_3$, in other words, we move $q_3$ through the transposition until it is forced into the braid move involving $\mathbf{r}$. Observe that there must not be any factors of $s_i^A$ in between $q_2$ and $r_2$ other than $q_1$, since $q_1$ does not cancel out as it commutes through. Moreover, since $q_1$ is the transposition which we are moving through our expression, it does not appear in $\bw$. Thus, in $\bw$, $q_2$ and $r_2$ satisfy the hypotheses of \Cref{lem:nobadbraid}. However, this contradicts \Cref{prop:distinguishedbraid}.

\end{proof}
The following lemma is easy to verify and we record it here for completeness.

\begin{lem}\label{lem:ycommutation}
If $s_i^As_j^A=s_j^As_i^A$ for $i\neq j$, then $g_i^Ag_j^A=g_j^Ag_i^A$, where $g_i^A\in\{y_i(t_i),\dot{s}_i, x_i(m_i)\dot{s}_i^{-1}\}$.    
\end{lem}

\begin{proof}
    This follows from the observations that $s_i^As_j^A=s_j^As_i^A$ if and only if $|i-j|>1$, and that $g_i^A$ is block diagonal with a $2\times 2$ block in rows $i,i+1$ and identity blocks in all other rows. 
\end{proof}

\begin{proof}[Proof of \Cref{thm:main} \ref{mainthm:consec}$\implies$\ref{mainthm:pos}]
In this proof, $G$ will denote either $\mathrm{Sp}_{2n}$ or $\mathrm{SO}_{2n+1}$, and $\mathrm{GFl}_{K}$ will denote the corresponding flag variety, $\mathrm{SpFl}_{K;2n}$ or $\mathrm{SOFl}_{K;2n+1}$, respectively. We will also denote by $\mathrm{Fl}_{K}$ the flag variety $\mathrm{Fl}_{K;2n}$ or $\mathrm{Fl}_{K;2n+1}$, respectively.

    The following lemma guarantees $\mathrm{GFl}_{K}^{>0}\subseteq\mathrm{GFl}_{K}^{\Delta>0}$ (in fact, for any $K\subset[n]$):

    \begin{lem}
    $\mathrm{Sp}_{2n}^{>0}\subseteq\mathrm{GL}_{2n}^{>0}$ and $\mathrm{SO}_{2n+1}^{>0}\subseteq\mathrm{GL}_{2n+1}^{>0}$.
    \end{lem}
    \begin{proof}
    This is because $(U_-^\Phi)^{>0}\subseteq(U_-^A)^{>0}$, $(T^\Phi)^{>0}\subseteq(T^A)^{>0}$, and $(U_+^\Phi)^{>0}\subseteq(U_+^A)^{>0}$.
    \end{proof}

 Consider a flag $F$ in $\mathrm{GFl}_{J}^{\Delta >0}$. Since all of its Pl\"ucker coordinates are positive and in particular nonzero, $F$ lies in the top Richardson cell. Thus, the Deodhar decomposition of $\mathrm{GFl}_J$ guarantees that
$F$ can be represented uniquely by a matrix $M=\mathbf{g}^\Phi_{\mathbf{i}}$, where $\mathbf{s}_\mathbf{i}=(\mathbf{w_0}^J)^\Phi$ and $g^\Phi_{i_k}\in \{y^\Phi_{i_k}(t_k),\dot{s}^\Phi_{i_k}, x^\Phi_{i_k}(m_k)(\dot{s}^\Phi_{i_k})^{-1}\}$. We can also view $F$ as a flag in $\mathrm{Fl}_{K'}^{\Delta > 0}=\mathrm{Fl}_{K'}^{>0}$ by \Cref{thm:BK}. Thus, as a flag of rank $K'$, $F$ can also be represented uniquely by a matrix $\mathbf{y}_{(\mathbf{w_0}^{J'})^A}(\mathbf{b})$ where $\mathbf{b}$ is such that each $b_i>0$. Let $\mathbf{s}_\mathbf{z}=(\mathbf{w_0}^{J'})^A$. We will show that this implies that for each $k$, $g^\Phi_{i_k}=y^\Phi_{z_k}(t_k)$ with $t_k>0$. 

Recall that $G$ satisfies properties \hyperref[condition1]{($\dagger$)} and \hyperref[condition2]{($\ddagger$)}. Using properties \hyperref[condition1]{($\dagger1$)}, \hyperref[condition2]{($\ddagger1$)} and \hyperref[condition2]{($\ddagger2$)}, we have that $\mathcal{R}^\Phi_{\bu,\bv}\subseteq \mathcal{R}^A_{\tilde{\psi}(\bu),\psi(\bv)}$. Thus, we can rewrite $M$ as $M=\mathbf{g}^A_{\mathbf{j}}$, where $\mathbf{s}_{\mathbf{j}}=\psi((\mathbf{w_0}^J)^\Phi)$ and $g^A_{j_k}\in \{y^A_{j_k}(t'_k),\dot{s}^A_{j_k}, x^A_{j_k}(m'_k)(\dot{s}^A_{j_k})^{-1}\}$. By \hyperref[condition2]{($\ddagger3$)} the positivity of the $\{t_k\}$ is equivalent to the positivity of the $\{t'_k\}$.

One may check by straightforward computations that if $s_{\alpha_1}^As_{\alpha_2}^As_{\alpha_3}^A=s_{\beta_1}^As_{\beta_2}^As_{\beta_3}^A$ are related by a braid move and $g^A_{\alpha_k}\in \{y^A_{\alpha_k}(t_k),\dot{s}^A_{\alpha_k}, x^A_{\alpha_k}(m_k)(\dot{s}^A_{\alpha_k})^{-1}\}$, then $g_{\alpha_1}^Ag_{\alpha_2}^Ag_{\alpha_3}^A={g'}_{\beta_1}^A{g'}_{\beta_2}^A{g'}_{\beta_3}^A$ for some choice of $\{t'_k,m'_k\}$ depending on $\{t_k,m_k\}$ and some choice of $g'^A_{\beta_k}\in \{y^A_{\beta_k}(t'_k),\dot{s}^A_{\beta_k}, x^A_{\beta_k}(m'_k)(\dot{s}^A_{\beta_k})^{-1}\}$.

By the observation in the previous paragraph and \Cref{lem:ycommutation}, together with \Cref{prop:braidmovesbiggerthann}, we can rewrite $M=\mathbf{g}^A_{\mathbf{j}}$ as $M=\mathbf{g}^A_{\mathbf{l}}\mathbf{g}^A_{\mathbf{p}}$, where $\mathbf{s}_\mathbf{p}=\bu$, $\mathbf{s}_\mathbf{l}=\left(\mathbf{w_0}^{J'}\right)^A$ and each $g^A_{l_k}\in \{y^A_{l_k}(t''_k),\allowbreak \dot{s}^A_{l_k}, x^A_{l_k}(m''_k)(\dot{s}^A_{l_k})^{-1}\}$ for some new choice of variables $\{t''_k,m''_k\}$. However, \Cref{prop:braidmovesbiggerthann} and \Cref{lem:ycommutation} guarantee that whenever $l_k<n$ and $g_{l_k}^A=y_{l_k}^A(t''_k)$, we have $t''_k=t'_{k'}$ for some $k'$. 

By \Cref{cor:Bruhatcomparison}, $\mathbf{g}^A_{\mathbf{j}}$ and $\mathbf{g}^A_{\mathbf{l}}$ represent the same flag of rank $K'$. We saw earlier that this same flag can be represented by a matrix $\mathbf{y}_{(\mathbf{w_0}^{J'})^A}(\mathbf{b})$ where $\mathbf{b}$ is such that each $b_k>0$. By the uniqueness of representatives in the Deodhar decomposition, we must have that each $g^A_{l_k}=y^A_{z_k}(t''_k)$ and that each $t''_k=b_k$. 

By \Cref{prop:keepsmalltranspositions}, $\mathbf{g}^A_{\mathbf{l}}$ contains a term originating from each $g_{i_k}^\Phi$ appearing in $\mathbf{g}^\Phi_{\mathbf{i}}$. Thus, for each $k$, we have $g_{i_k}^\Phi=y^\Phi_{i_k}(t_k)$, where $b_k=f_j(t_k)$ for some $f_j$ appearing in condition \hyperref[condition1]{($\dagger1$)}. In particular, by \hyperref[condition1]{($\dagger1$)} and \hyperref[condition2]{($\ddagger3$)}, $t_k>0$. This completes the proof.

\end{proof}

\small
\bibliography{ref.bib}
\bibliographystyle{alpha}
\end{document}